\documentclass[11pt,reqno]{amsart}
\usepackage{amsmath, graphicx}
\usepackage{amsfonts, amssymb}
\usepackage[english]{babel}
\usepackage{graphicx}
\providecommand{\U}[1]{\protect\rule{.1in}{.1in}}

\numberwithin{equation}{section}
\newtheorem{theorem}{Theorem}[section]
\newtheorem{corollary}[theorem]{Corollary}
\newtheorem{lemma}[theorem]{Lemma}
\newtheorem{proposition}[theorem]{Proposition}
\theoremstyle{definition}
\newtheorem{definition}[theorem]{Definition}
\newtheorem{remark}[theorem]{Remark}
\newtheorem{example}[theorem]{Example}
\begin{document}

\begin{center}
{\LARGE Metric and geometric relaxations of self-contracted curves }

\vspace{0.7cm}

{\large \textsc{A. Daniilidis, R. Deville, E. Durand-Cartagena}}

\bigskip
\end{center}

\noindent\textbf{Abstract} Self-contractedness (or self-expandedness,
depending on the orientation) is hereby extended in two natural ways giving
rise, for any $\lambda\in\lbrack-1,1)$, to the metric notion of $\lambda
$-curve and the (weaker) geometric notion of $\lambda$-cone property
($\lambda$-eel). In the Euclidean space $\mathbb{R}^{d}$ it is established
that for $\lambda\in\lbrack-1,1/d)$ bounded $\lambda$-curves have finite
length. For $\lambda\geq 1/\sqrt{5}$ it is always possible to construct bounded
curves of infinite length in ${\mathbb{R}}^{3}$ which do satisfy the $\lambda
$-cone property. This can never happen in ${\mathbb{R}}^{2}$ though: it is
shown that all bounded planar curves with the $\lambda$-cone property have
finite length.

\vspace{0.5cm}


\vspace{0.5cm}

\textit{Secondary} 37N40, 53A04, 53B20, 52A41.

\tableofcontents

%%%%%%%%%%%%
% Section 1: Introduction

\section{Introduction}

Self-contracted curves have been introduced in \cite{DLS}. They attract a lot
of interest, since they are intimately linked to convex foliations
(\cite{CBM}, \cite{LMV}, \cite{MP}), to the proximal algorithm of a convex
function and the gradient flow of a quasiconvex potential in a Euclidean space
(\cite{DLS}, \cite{DDDL}) and recently to generalized flows in CAT(0) spaces
(\cite{OH}). The main feature of this notion is its simple purely metric definition, which
inspires developments in more general settings:

\begin{definition}
Let $(M,d)$ be a metric space and $I\subset\mathbb{R}$ be an interval. A curve
$\gamma:I\rightarrow M$ is called self-contracted, if for all $\tau\in I$, the
map $t\mapsto d(\gamma(t),\gamma(\tau))$ is non-increasing on $I\cap
(-\infty,\tau]$.
\end{definition}

The {length} of a curve $\gamma$ is defined as
\[
\ell(\gamma):=\sup\Big\{\sum_{i=0}^{m-1}d(\gamma(t_{i}),\gamma(t_{i+1}
))\Big\},
\]
where the supremum is taken over all finite increasing sequences $t_{0}<t_{1}<\cdots<t_{m}$ lying in $I$. 
The curve $\gamma$ is called {rectifiable},
if its total variation is locally bounded around any $t\in I$, that is, its
length is locally finite. \smallskip

Rectifiability and asymptotic behaviour are central questions in the study of
self-contracted curves. It is shown in \cite{DLS} that self-contracted curves
(are rectifiable and) have finite length whenever $M$ is a bounded subset of
the 2-dimensional Euclidean space. Based on ideas of \cite{MP}, the
aforementionned result was extended in \cite{DDDL}, and independently in
\cite{LMV}, to any finite dimensional Euclidean space. In \cite{DDDR} a
further extension has been established encompassing the case where $M$ is a
compact subset of a Riemannian manifold. In \cite{L} the result of \cite{DLS}
has been generalized for 2-dimensional spaces equipped with other (smooth)
norms. This has been the first result of this type outside a
Euclidean/Riemannian setting. An important breakthrough is eventually achieved
in \cite{ST} by establishing (rectifiability and) finite length for all
self-contracted curves contained on a bounded subset of any finite dimensional
normed space. Finally, rectifiability of self-contracted curves in Hadamard
manifolds and CAT(0) spaces is established in \cite{OH}.

\smallskip

The aforementioned results remain valid if we replace the assumption
\textquotedblleft$\gamma$ self-contracted\textquotedblright\ by the assumption
\textquotedblleft$\gamma$ self-expanded\textquotedblright. A curve $\gamma$ is
called \emph{self-expanded} if for all $\tau\in I$, the map $t\mapsto
d(\gamma(t),\gamma(\tau))$ is non decreasing on $I\cap\lbrack\tau,+\infty)$,
or equivalently, when the curve $\overline{\gamma}:-I\rightarrow M$ given by
$\overline{\gamma}(t)=\gamma(-t)$ is self-contracted. Thus,
$\gamma:I\rightarrow{\mathbb{R}}^{d}$ is self-expanded 
if for every $t_{1}\leq t_{2}\leq t_{3}$ in $I$ we have
\[
d(\gamma(t_{1}),\gamma(t_{2}))\leq d(\gamma(t_{1}),\gamma(t_{3})). 
\]
In the Euclidean
setting, there is a nice geometric interpretation of self-expandedness (see
\cite[Lemma 2.8]{DDDL}). A differentiable curve is self-expanded if and only
if
\[
\langle\gamma^{\prime}(t),\gamma(u)-\gamma(t)\rangle\leq0\ \hbox{
for all $u\in I$ such that $u < t$},
\]
which geometrically means that the tail of the curve (the past) is always
contained in half-space (cone of aperture $\pi$).
The notion of self-expandedness therefore admits the following two natural generalizations.
Let us fix $-1\leq\lambda<1$. 
A curve $\gamma:I\rightarrow{\mathbb{R}}^{d}$ is called $\lambda$-curve  
if for every $t_{1}\leq t_{2}\leq t_{3}$ in $I$ we have
\begin{equation}
d(\gamma(t_{1}),\gamma(t_{2}))\leq d(\gamma(t_{1}),\gamma(t_{3}))+\lambda
d(\gamma(t_{2}),\gamma(t_{3})). \label{l-curve}
\end{equation}
If $\gamma$ is continuous and admits right derivative at each point,
we say that $\gamma$ has the $\lambda$-cone-property if,
for every $t<\tau$ in $I$, we have, denoting $\gamma'(\tau)$ the right derivative,
\[
\langle\gamma'(\tau),\gamma(t)-\gamma(\tau)\rangle\leq\lambda
\,\vert\vert\gamma'(\tau)\vert\vert \,\left\vert
\left\vert \gamma(t)-\gamma(\tau)\right\vert \right\vert .
\]
As a matter of the fact, the $\lambda$-cone property will be defined more generally, for merely continuous curves 
using (forward) secants, see Definition~\ref{def_l-cone-prop} and it will be shown that every $\lambda$-curve has the $\lambda$-cone property (\textit{c.f.} Proposition~\ref{prop-implies}). However there exist smooth curves satisfying the
latter property for some $\lambda_{0}<1$ without being $\lambda$-curves for
any $\lambda\in\lbrack-1,1)$ (\textit{c.f.} Example~\ref{ex3dcurve}).
\smallskip\newline In this work we establish the following results: \smallskip

\begin{itemize}
\item if $|\cdot|$ is an equivalent norm to the Euclidean norm $||\cdot||$,
then there exists $\lambda\in\lbrack0,1)$ such that every $|\cdot
|$-self-expanded curve is a $||\cdot||$-$\lambda$-curve
(Proposition~\ref{Prop-norm-equiv}); \smallskip

\item for $\lambda<1/d$ every bounded $\lambda$-curve (is rectifiable and) has
finite length (Theorem~\ref{Thm_l-curve-rect});\smallskip

\item for $\lambda\geq1/\sqrt{5}$ there exists a bounded curve in ${\mathbb{R}
}^{3}$ with infinite length satisfying the $\lambda$-cone property
(Theorem~\ref{main}).\smallskip
\end{itemize}

Nonetheless due to topological obstructions for $d=2$ we have:

\begin{itemize}
\item for any $\lambda<1$, bounded planar curves with the $\lambda$-cone
property (and a fortiori $\lambda$-curves) have finite length
(Theorem~\ref{prop-2-dim}). \smallskip
\end{itemize}

Combining the first and the last statement, we readily obtain that all bounded
planar self-contracted curves (under any norm) are rectifiable and have finite
length. This clearly generalizes the result of \cite{L}, but it is contained
in the result of \cite{ST} that asserts that the same holds in any dimension.
Notice that the asymptotic behaviour of both $\lambda$-curves and curves with
the $\lambda$-cone property remains unknown in $\mathbb{R}^{d}$ for $d\geq3$
and $\lambda\in [1/d,1/\sqrt{5}).$

\medskip

\textbf{Notation.} Let us fix our notation. Throughout this work ${\mathbb{R}}^{d}$ 
will denote the $d$-dimensional Euclidean space endowed with the
Euclidean norm $||\cdot||$ and the scalar product $\langle\cdot,\cdot\rangle$.
We denote by $\mathbb{S}^{d-1}$ the unit sphere of ${\mathbb{R}}^{d}$, and by
$B(x,r)$ (respectively, $\overline{B}(x,r)$) the open (respectively, closed)
ball of radius $r>0$ and center $x\in{\mathbb{R}}^{d}$. A (convex) subset $C$
of $\mathbb{R}^{d}$ is called a (convex) cone, if for every $x\in C$ and $r>0$
it holds $rx\in C$. If $A$ is a nonempty subset of ${\mathbb{R}}^{d}$, we
denote by $\mathrm{int}(A)$ its interior, by $\mathrm{conv\,}(A)$ its {convex
hull} and by $\text{diam}\,A:=\sup\,\{d(x,y):x,y\in A\}$ its {diameter}.\smallskip

Given a closed convex subset $K$ of ${\mathbb{R}}^{d}$, the normal cone
$N_{K}(u_{0})$ of $K$ at $u_{0}\in K$ is the following closed convex cone (see
\cite{Rock98} \textit{e.g.}):
\[
N_{K}(u_{0})=\{v\in\mathbb{R}^{n}:\langle v,u-u_{0}\rangle\leq0,\forall u\in
K\}.
\]
Notice that $u_{0}\in K$ is the projection onto $K$ of all elements of the
form $u_{0}+tv$, where $t\geq0$ and $v\in N_{K}(u_{0})$. In the particular
case that $K$ is a closed convex pointed cone (that is, $K$ contains no
lines), then its polar (or dual) cone 
\[
K^{o}:=N_{K}(0)=\{v\in\mathbb{R}^{n}:\langle v,u\rangle\leq0,\forall u\in
K\}
\]
has nonempty interior and the bipolar theorem holds: $K^{oo}=K.$ For $\delta>0$ sufficiently small,
we denote by $K_{\delta}$ the $\delta$-enlargement of the cone $K$, that is,
the closed convex cone generated by the set $\left(  K\cap \mathbb{S}^{d-1}\right)
+B_{\delta}$, where $B_{\delta}:=B(0,\delta).$ Notice that
\begin{equation}
\left((K_{\delta})^{o}\cap \mathbb{S}^{d-1}\right)  +B_{\delta
}\,\,\subset\,K^{o}. \label{polar}
\end{equation}

\smallskip

We define the {aperture} $A(S)$ of a nonempty subset $S\subset\mathbb{S}
^{d-1}$ by
\begin{equation}
A(S):=\inf\left\{  \,\langle u_{1},u_{2}\rangle\,:\,u_{1},u_{2}\in
S\,\right\}  . \label{aperture}
\end{equation}
Based on the above notion, we define the aperture $\mathcal{A}(C)$ of a
nontrivial convex pointed cone $C$ as follows:
\[
\mathcal{A}(C)=\arccos\,\left(  A(C\cap \mathbb{S}^{d-1})\right)  .
\]
Given $v\in\mathbb{S}^{d-1}$ and $\alpha\in\lbrack0,\pi)$, we define the
\textquotedblleft open" cone directed by $v$ as follows:
\begin{equation}
C(v,\alpha)=\left\{  u\in\mathbb{R}^{d}:\,\,\langle u,v\rangle>\,||u||\,\cos
\alpha\right\}  \cup\{0\}. \label{open-cone}
\end{equation}
Notice that if $\alpha<\pi/2$, the above cone is convex and has aperture~$2\alpha$. Given $x\in{\mathbb{R}}^{d}$, we adopt the notation
\begin{equation}
C_{x}(v,\alpha):=x+C(v,\alpha). \label{cone-notation}
\end{equation}

\smallskip

A mapping $\gamma:I=[0,T_{\infty})\rightarrow{\mathbb{R}}^{d}$, where
$T_{\infty}\in\mathbb{R\cup\{+\infty\}}$ is referred in the sequel as a
{curve}. Although the usual definition of a curve comes along with continuity
and injectivity requirements for the map $\gamma$, we do not make these prior
assumptions here. By the term {continuous} (respectively, {absolutely
continuous}, {Lipschitz}, {smooth}) curve we shall refer to the corresponding
properties of the mapping $\gamma:I\rightarrow{\mathbb{R}}^{d}$. A curve
$\gamma$ is said to be {bounded} if its image, denoted by $\Gamma=\gamma(I)$,
is a bounded set of ${\mathbb{R}}^{d}$.

\smallskip

For $t\in I$ we denote by $\Gamma(t):=\{\gamma(t^{\prime})\in\Gamma:t^{\prime
}\leq t\}$ the initial part of the curve and by
\begin{equation}
K(t)=\overline{\mathrm{cone}}\mathrm{\,}(\Gamma(t)-\gamma(t)) \label{K(t)}
\end{equation}
the closed convex cone generated by $\Gamma(t)$. In particular
\begin{equation}
\Gamma(t)\subset\gamma(t)+K(t) \label{KG}
\end{equation}
Notice further that $K(t)$ contains the set $\mathrm{sec}^{-}(t)$ of (all
possible limits of) {backward secants} at $\gamma(\tau)$ which is defined as
follows (see \cite{DDDL}):
\[
\mathrm{sec}^{-}(t):=\left\{  q\in\mathbb{S}^{d-1}:q=\lim_{t_{k}\nearrow
t^{-}}\,\frac{\gamma(t_{k})-\gamma(t)}{||\gamma(t_{k})-\gamma(t)||}\right\}
,\label{sec-}
\]
where the notation $\{t_{k}\}_{k}\nearrow t^{-}$ indicates that $\{t_{k}
\}_{k}\rightarrow t$ and $t_{k}<t$ for all $k$. \smallskip

The set $\mathrm{sec}^{+}(t)$ of all possible limits of {forward secants} at
$\gamma(t)$ is defined analogously:
\[
\mathrm{sec}^{+}(t):=\left\{  q\in\mathbb{S}^{d-1}:q=\lim_{t_{k}\searrow
t^{+}}\,\frac{\gamma(t_{k})-\gamma(t)}{||\gamma(t_{k})-\gamma(t)||}\right\}
,\label{sec+}
\]
where the notation $\{t_{k}\}_{k}\searrow t^{+}$ indicates that $\{t_{k}
\}_{k}\rightarrow t$ and $t<t_{k}$ for all $k$. Compactness of $\mathbb{S}^{d-1}$ guarantees that both $\mathrm{sec}^{-}(t)$ and $\mathrm{sec}^{+}(t)$
are nonempty. If $\gamma:I\rightarrow\mathbb{R}^{d}$ is differentiable at
$t\in I$ and $\gamma^{\prime}(t)\neq0$, then $\mathrm{sec}^{+}(t)=\left\{
\frac{\gamma^{\prime}(t)}{||\gamma^{\prime}(t)||}\right\}  $.\medskip

In this work we introduce two new notions, depending on a parameter
$\lambda\in\lbrack-1,1)$. For each value of $\lambda$ we obtain
the class of $\lambda$-curves and the class of curves with the $\lambda$-cone
property. We associate to these classes an angle $\alpha\in(0,\pi]$ via the
relation
\begin{equation}
\alpha=\arccos(\lambda). \label{a(l)}
\end{equation}
As we shall see, the above classes enjoy interesting geometric properties
which can be described in terms of the angle $\alpha$. (For $\lambda=0$, which
corresponds to the angle $\alpha=\pi/2$, the above classes coincide and yield
the class of self-expanded curves.)

%%%%%%%%%%%%%%%%%%%%%%%
%Section 2: The classe of $\lambda$-curves

\section{$\lambda$-curves and curves with the $\lambda$-cone property}

\begin{definition}
[$\lambda$-curve]\label{def_l-curve} A curve $\gamma:I\rightarrow{\mathbb{R}}^{d}$ is called $\lambda$-curve $(-1\leq\lambda<1)$ if for every $t_{1}\leq
t_{2}\leq t_{3}$ in $I$ we have
\begin{equation}
d(\gamma(t_{1}),\gamma(t_{2}))\leq d(\gamma(t_{1}),\gamma(t_{3}))+\lambda
d(\gamma(t_{2}),\gamma(t_{3})). \label{l-curve}
\end{equation}

\end{definition}

The above definition yields that every $\lambda$-curve is necessarily injective and
cannot admit more than one accumulation point. Based on this, one can easily
see that every $\lambda$-curve has at most countable discontinuities. Setting
$\lambda=1$ to \eqref{def_l-curve} yields the triangle inequality of the
distance (hence no restriction) while $\lambda=-1$ corresponds to segments. On
the other hand, for $\lambda=0$ we recover the definition of a self-expanded
curve. The following result shows that the study of
self-contracted/self-expanded curves with respect to a non-Euclidean norm can
be shifted to the study of $\lambda$-curves in the Euclidean setting.

\begin{proposition}
[self-expanded vs $\lambda$-curve]\label{Prop-norm-equiv}Let $\|\cdot\|$ be
an Euclidean norm in $\mathbb{R}^{d}$ and $|\cdot|$ be another norm in $\mathbb{R}^{d}$. 
Then there exists $\lambda<1$
(depending on the equivalence constant of the norms) such that every $|\cdot
|$-self-expanded curve is a $\|\cdot\|$-$\lambda$-curve.
\end{proposition}

\noindent\textit{Proof.} Since the norms $|\cdot|$ and $\Vert\cdot\Vert$ are
equivalent and since the properties of being self-expanded or being a
$\lambda$-curve are invariant by homothetic transformation, we may assume that
there exists $\delta>0$ such that for all $x\in\mathbb{R}^{d}$, $\delta\Vert
x\Vert\leq|x|\leq\Vert x\Vert$. Let $t_{0}<t_{1}<t_{2}$ in $I$ and set
$x_{0}=\gamma(t_{0})$, $x=\gamma(t_{1})$ and $y=\gamma(t_{2})$. It follows by
assumption that
\begin{equation}
|x-x_{0}|\leq|y-x_{0}|. \label{self}
\end{equation}
To establish the result it is sufficient to prove that there exists
$\lambda<1$ such that for all choices of $x_{0},x,y$ satisfying \eqref{self},
we have
\[
\Vert x-x_{0}\Vert\leq\Vert y-x_{0}\Vert+\lambda\Vert x-y\Vert.
\]
By translation, we may, and do assume, that $x_{0}=0$. Moreover, by
homogeneity, we can assume $|y|=1$. Set
\[
B=\{z\in\mathbb{R}^{d};\,|z|\leq\delta\}\quad\text{and\quad}C_{y}
=\{y+t(y-z);\,\Vert z\Vert<\delta,\,t>0\}.
\]
\smallskip

We claim that $B\cap C_{y}=\emptyset$. Indeed, fix $z$ such that $\Vert
z\Vert<\delta$. The function $\varphi:\mathbb{R}\to\mathbb{R}$ defined by
$\varphi(t)=\vert y+t(y-z)\vert$ is convex, $\varphi(-1)<1$ and $\varphi
(0)=1$, hence $\varphi(t)>1$ whenever $t>0$, that is, $y+t(y-z)\notin B$.
Since this is true for all $z$ satisfying $\Vert z\Vert<\delta$, the claim is
proved. \smallskip\noindent

Therefore Proposition~\ref{Prop-norm-equiv} is a consequence of the following lemma.

\begin{lemma}
There exists $\lambda<1$ such that, whenever $1\le\Vert y\Vert\le1/\delta$,
$\Vert x\Vert\le1/\delta$ and $x\notin C_{y}$, then $\Vert x\Vert-\Vert
y\Vert\le\lambda\Vert x-y\Vert$.
\end{lemma}

\medskip\noindent\textit{Proof.} Set $u=x-y$ and $\Gamma_{y}
:=\{t(y-z);\,||z||<\delta,\,t>0\}$. We claim that there exists $\rho<1$ such
that, whenever $1\leq||y||\leq1/\delta$ and $u\in\mathbb{R}^{d}\diagdown
\Gamma_{y}$, then
\begin{equation}
\langle u,y\rangle\leq\rho\,||u||\cdot||y||. \label{cos}
\end{equation}

Indeed, since $\mathbb{R}^{d}\diagdown\Gamma_{y}$ is a cone, it is enough to
establish (\ref{cos}) when $\Vert u\Vert=\Vert y\Vert$. Let us denote by
$c(u,y)$ the cosine of the angle of the two vectors $u$ and $y$. The condition
$u\notin\Gamma_{y}$ yields $\Vert u-y\Vert\geq\delta$. Then we obtain $\Vert
u-y\Vert^{2}=\Vert y\Vert^{2}(2-2c(u,y))\geq1$, which yields
\[
c(u,y)\,\leq\,1-\frac{\delta^{2}}{2\Vert y\Vert^{2}}\,\leq\,1-\frac{\delta
^{4}}{2}:=\rho\,<\,1.
\]
This proves the claim. \medskip\noindent

Since $||y+u||^{2}\leqslant||y||^{2}+||u||^{2}+2||y||\,||u||\,\rho$, we deduce
from \eqref{cos} that
\[
||x||-||y||=||y+u||-||y||\,\leq\,||y||\left(  \sqrt{1+\frac{2\rho||u||}
{||y||}+\frac{||u||^{2}}{||y||^{2}}}-1\right)  .
\]
Since $||u||=||x-y||\leq||x||+||y||\leq2/\delta$ and $||y||\geq|y|=1$, we have
$t=\frac{||u||}{||y||}\in\lbrack0,2/\delta]$. Notice that taking $\lambda<1$
sufficiently close to 1, we ensure that for all $t\in\lbrack0,2/\delta]$ it
holds
\[
\sqrt{1+2\rho t+t^{2}}-1\leq\lambda t.
\]
Therefore we conclude that
\[
||x||-||y||\leq\lambda||u||=\lambda||x-y||.
\]
The proof is complete. \hfill$\square$

\bigskip

From now on we consider exclusively a Euclidean setting. An important feature
of the notion of $\lambda$-curve is the following property:

\begin{proposition}
[uniform non-collinearity]\label{prop-cola}Let $\gamma:I\rightarrow
{\mathbb{R}}^{d}$ be a $\lambda$-curve. Then, $\gamma$ is $\lambda$-uniformly
non-collinear, that is, for every $s,u,t\in I$ such that $s,u\leq t$ we have
\begin{equation}
\Big\langle\dfrac{\gamma(u)-\gamma(t)}{\Vert\gamma(u)-\gamma(t)\Vert}
,\dfrac{\gamma(s)-\gamma(t)}{\Vert\gamma(s)-\gamma(t)\Vert}
\Big\rangle\,>\,-\lambda\quad\left(  \,>\,-1\,\right)  \,. \label{cola}
\end{equation}

\end{proposition}

\noindent\textit{Proof.} Assume that $u<s<t$. Because $\gamma$ is $\lambda
$-curve we have that
\[
d(\gamma(u),\gamma(s))\,\leq\,d(\gamma(u),\gamma(t))\,+\,\lambda
\,d(\gamma(s),\gamma(t))
\]
Consider the triangle of vertices $\gamma(t)$, $\gamma(u)$ and $\gamma(s)$ and
set $c=d(\gamma(u),\gamma(s))$, $a=d(\gamma(u),\gamma(t))$ and $b=d(\gamma
(s),\gamma(t))$. The previous equation now reads $c\leq a+\lambda b$, and
after squaring both sides we get%
\begin{equation}
c^{2}\leq a^{2}+\lambda^{2}b^{2}+2\lambda ab. \label{eqsquare}
\end{equation}
Evoking the law of cosine $c^{2}=a^{2}+b^{2}-2ab\cos\varphi$ we deduce
\[
\cos\varphi=\frac{a^{2}+b^{2}-c^{2}}{2ab}\,\overset{\eqref{eqsquare}}{\geq
}\,\frac{(1-\lambda^{2})b^{2}-2\lambda ab}{2ab}\,>\,-\lambda,
\]
that is, the angle $\varphi$ between the vectors
\[
\dfrac{\gamma(u)-\gamma(t)}{\Vert\gamma(u)-\gamma(t)\Vert}\quad\text{and}
\quad\dfrac{\gamma(s)-\gamma(t)}{\Vert\gamma(s)-\gamma(t)\Vert}
\]
is strictly less than $\pi-\alpha$ ($\alpha=\arccos(\lambda)$ is given by
\eqref{a(l)}).\hfill$\square$

\bigskip

Before we proceed, we give the following definition.

\begin{definition}
[$\lambda$-cone property]\label{def_l-cone-prop} Let $\lambda\in\lbrack-1,1)$
and $\alpha=\arccos(\lambda)$. We say that a continuous curve $\gamma
:I\rightarrow{\mathbb{R}}^{d}$ satisfies the $\lambda$-cone property if for
every $t\in I$ and for every $q_{t}^{+}\in\mathrm{sec}^{+}(t)$ it holds
\begin{equation}
\left\langle q_{t}^{+},\frac{\gamma(u)-\gamma(t)}{\Vert\gamma(u)-\gamma
(t)\Vert}\right\rangle \leq\lambda,\quad\text{for all }u<t. \label{sec-eel}
\end{equation}

\end{definition}

In other words, recalling \eqref{open-cone}, the set $\Gamma(t)-\gamma(t)$
does not intersect the cone $C\left(  q_{t}^{+},\alpha\right)  $ directed by
$q_{t}^{+}$ and of aperture $2\alpha$ expect at $0$, that is, for every $t\in
I$
\begin{equation}
\left(  {\gamma(t)\,+\bigcup\limits_{q_{t}^{+}\in\mathrm{sec}^{+}(t)}}C\left(
q_{t}^{+},\alpha\right)  \right)  \,\,\bigcap\,\;\Gamma(t)\,=\,\left\{
\gamma(t)\right\}  . \label{l-sec-snake}
\end{equation}

We shall now consider a second important feature of the class of (continuous)
$\lambda$-curves.

\begin{proposition}
[$\lambda$-curve $\Longrightarrow$ $\lambda$-cone property]
\label{prop-implies}Every continuous $\lambda$-curve has the $\lambda$-cone property.
\end{proposition}

\noindent\textit{Proof.} Fix $t\in I$, let $u<t$, $q_{t}^{+}\in\mathrm{sec}
^{+}(t)$ and choose $\{t_{k}\}_{k}\searrow t$ such that
\[
\frac{\gamma(t_{k})-\gamma(t)}{\Vert\gamma(t_{k})-\gamma(t)\Vert
}\,\longrightarrow\,q_{t}^{+}.
\]
Since $\gamma$ is a $\lambda$-curve we have
\[
\Vert\gamma(t)-\gamma(u)\Vert\leq\Vert\gamma(t_{k})-\gamma(u)\Vert
+\lambda\Vert\gamma(t_{k})-\gamma(t)\Vert,
\]
yielding
\[
\frac{\Vert\gamma(t_{k})-\gamma(u)\Vert-\Vert\gamma(t)-\gamma(u)\Vert
}{||\gamma(t_{k})-\gamma(t)||}\,\geq\,-\lambda
\]
Set $\Phi(X)=||X||,$ $X_{k}=\gamma(t_{k})-\gamma(u)$ and $X=\gamma
(t)-\gamma(u).$ Then the above inequality reads
\[
\frac{\Phi(X_{k})-\Phi(X)}{||X_{k}-X||}\,\geq\,-\lambda.
\]
Since the norm is differentiable around the segment $[X,X_{k}]:=\{tX+(1-t)X_k:\,t\in [0,1]\}$, applying the
Mean Value theorem we obtain $\theta_{k}\in\lbrack0,1)$ such that
\[
\Phi(X_{k})-\Phi(X)=D\Phi(X+\theta_{k}(X_{k}-X))(X_{k}-X)=\left\langle
\frac{X+\theta_{k}(X_{k}-X)}{\Vert X+\theta_{k}(X_{k}-X)\Vert},X_{k}
-X\right\rangle .
\]
Combining the above formulas and taking the limit as $k\rightarrow\infty$ we
get
\[
\left\langle \frac{\gamma(t)-\gamma(u)}{\Vert\gamma(t)-\gamma(u)\Vert}
,\,q_{t}^{+}\right\rangle \,\geqslant\,-\lambda.
\]
The above is equivalent to (\ref{sec-eel}) and the proof is complete.\hfill
$\square$

\bigskip

The following example reveals that there exist $C^{1}$curves satisfying the
$\lambda$-cone property but failing to satisfy the non-collinearity property.
Therefore these curves cannot be $\lambda$-curves for any value of the
parameter $\lambda\in\lbrack-1,1)$.

\begin{example}
\label{ex3dcurve} Let $\gamma:[-3\pi/2,1+\pi]\rightarrow{\mathbb{R}}^{3}$ be
defined by
\[
\gamma(t)\, =\left\{
\begin{array}
[c]{ll}%
(0,-\sin t,-\cos t), & \text{if }t\in\lbrack-3\pi/2,-\pi/2], \smallskip\\
(-\frac{1}{2}(1+\cos2t),1,\frac{1}{2}\sin2t), & \text{if }t\in\lbrack
-\pi/2,0], \smallskip\\
(-1,1,t), & \text{if }t\in\lbrack0,1],\smallskip\\
(-1,\frac{1}{2}(1+\cos2(t-1)),1+\frac{1}{2}\sin2(t-1)), & \text{if }
t\in\lbrack1,1+\pi/2],\smallskip\\
(-\sin(t-1),0,1+\cos(t-1)), & \text{if }t\in\lbrack1+\pi/2,1+\pi].
\end{array}
\right.
\]
It is easy to check that $\gamma$ is $C^{1}$-smooth. Moreover, $\gamma$ fails
to satisfy the non-collinearity property: indeed, $\gamma(1+\pi)=(0,0,0)$ is
the midpoint of the segment $[\gamma(-3\pi/2),\gamma(-\pi/2)]$. Hence, by
Proposition~\ref{prop-cola}, $\gamma$ cannot be a $\lambda$-curve for any
value of the parameter $\lambda<1$. On the other hand, any tangent line
\[
\{\gamma(t)+s\gamma^{\prime}(t);\,s\in\mathbb{R} \}
\]
meets the curve $\{\gamma(\tau);\,\tau\in[-3\pi/2,1+\pi]\}$ only at the point
$\gamma(t)$. Therefore, by a simple compactness argument, there exists
$\lambda_{0}<1$ for which $\gamma$ satisfies the $\lambda_{0}$-cone property.

\begin{figure}[h!]
\includegraphics[width=0.6\textwidth]{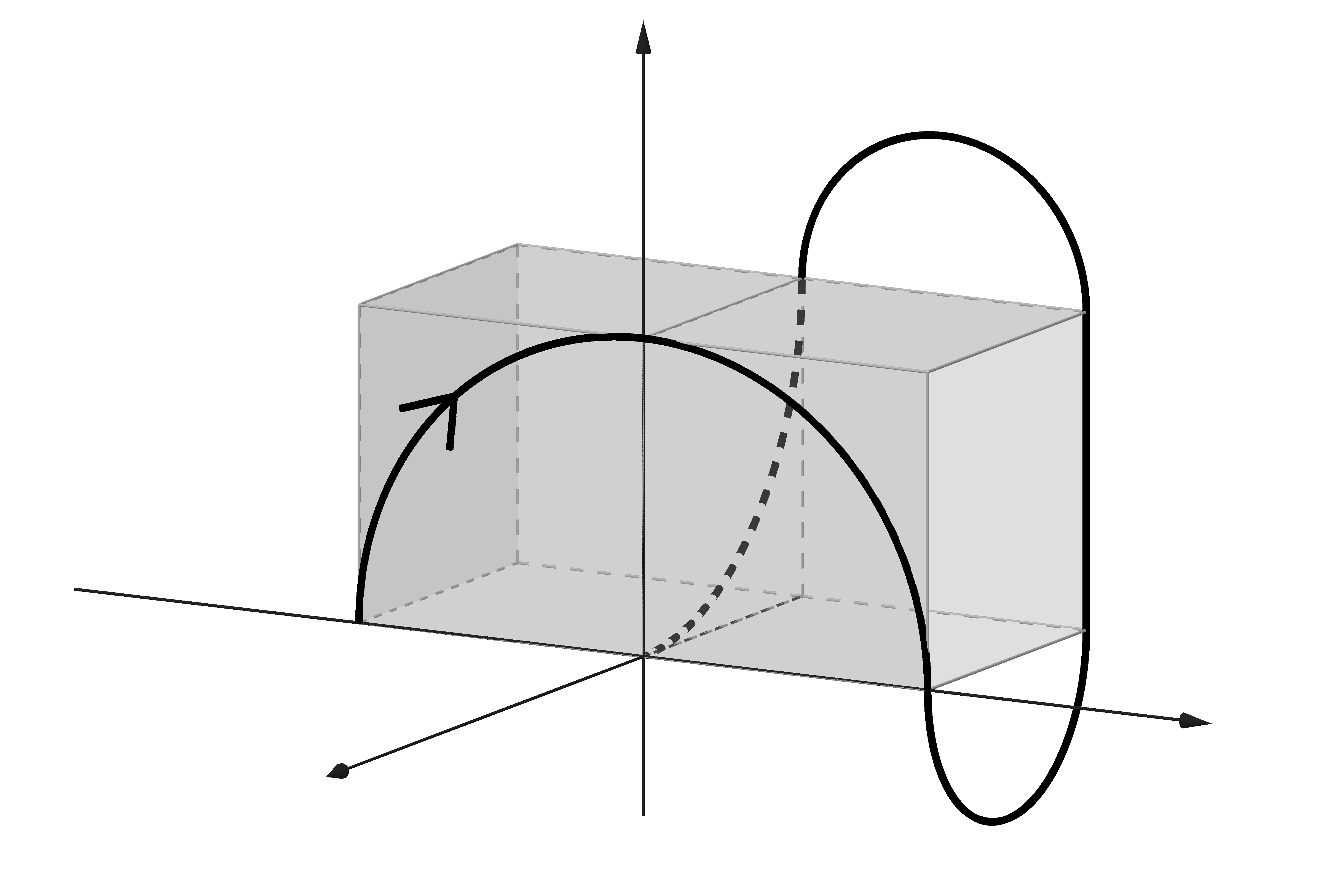}
\caption{Example of a curve with the $\lambda_0$-cone property, failing to be $\lambda$-curve for any $\lambda<1$. }
\label{circletangent}
\end{figure}
\end{example}

\section{Length of $\lambda$-curves}

Before we proceed we recall from \cite{DDDL} the following result (we provide
a proof for completeness).

\begin{lemma}
\label{Lem-2.5} Let $\Sigma\subset\mathbb{S}^{d-1}$ $($the unit sphere of
$\mathbb{R}^{d}$, $d>1)$ and assume that for $\lambda<1/d$ it holds
\[
\langle x,x^{\prime}\rangle\geq-\lambda,\quad\text{for all }x,x^{\prime}
\in\Sigma.
\]
Then $\Sigma$ is contained in a half-sphere (therefore it generates a closed
convex pointed cone).
\end{lemma}

\noindent\textit{Proof.} Notice that the conclusion holds if and only if
$0\notin\mathrm{conv\,}(\Sigma)$. Let us assume that $0\in\mathrm{conv\,}
(\Sigma).$ Then by Caratheodory theorem, there exist $\alpha_{0},\alpha
_{1},\cdots,\alpha_{d}\geq0$ and $x_{0},x_{1},\cdots,x_{d}\in\mathbb{S}^{d-1}$
such that
\[
\sum_{i=0}^{d}\alpha_{i}=1\quad\text{and}\quad\sum_{i=0}^{d}\alpha_{i}
x_{i}=\mathbf{0}.
\]
It follows that for $j\in\{0,1,\ldots,d\}$,
\[
0=\langle\mathbf{0},x_{j}\rangle=\sum_{i=0}^{d}\alpha_{i}\langle x_{i}
,x_{j}\rangle\geq\alpha_{j}-\lambda\sum_{i\neq j}\alpha_{i}=\alpha_{j}
-\lambda(1-\alpha_{j}).
\]
Summing up for all $j\in\{0,1,\ldots,d\}$ we get $0\geq1-\lambda(d+1-1),$
which contradicts the assumption $\lambda<1/d$.\hfill$\square$

\bigskip

Recalling the notation of \eqref{aperture}, \eqref{KG} and \eqref{a(l)}, and
assuming $\lambda<1/d$ we obtain the following result (as a straightforward
combination of Lemma~\ref{Lem-2.5} with Proposition~\ref{prop-cola}).

\begin{corollary}
[conical control of the initial part]\label{Cor-implies-}Let $-1\leq
\lambda<1/d$ and $\alpha=\arccos(\lambda).$ Then for every $t\in I,$ the
initial part $\Gamma(t)$ of a $\lambda$-curve $\gamma$ is contained in a
closed convex cone $K(t)$ of aperture at most $\pi-\alpha$ centered at $\gamma(t)$.
In other words,
\begin{equation}
\Gamma(t)\subset\gamma(t)+K(t)\quad\text{and\quad}\mathcal{A}(K(t))\leq
\pi-\alpha. \label{A(K(t))}
\end{equation}

\end{corollary}

To sum up, given a continuous $\lambda$-curve $\gamma$,
Proposition~\ref{prop-implies} ensures that its initial part $\Gamma(t)$
avoids the union of all cones centered at $\gamma(t)$ and directed by forward
secants of $\gamma$ at $t$, see \eqref{l-sec-snake}, while
Corollary~\ref{Cor-implies-} asserts that, provided $\lambda<1/d$, the initial
part of the curve $\Gamma(t)$ is itself contained in the closed convex pointed
cone $\gamma(t)+K(t)$, centered at $\gamma(t)$. The following proposition
asserts that an even stronger property is satisfied.

\begin{proposition}
[conical split at each $t$]\label{prop-gap} Let $\gamma:I\rightarrow
{\mathbb{R}}^{d}$ be a continuous $\lambda$-curve, with $\lambda\in
\lbrack-1,1/d)$ and $\alpha=\arccos(\lambda).$ Then it holds:
\begin{equation}\label{todo}
\left(  {\displaystyle\bigcup\limits_{q_{t}^{+}\in\mathrm{sec}^{+}(t)}
}C\left(  q_{t}^{+},\alpha\right)  \right)  \,\,\bigcap
\,\,\,K(t)\,\,=\,\,\{0\}, \qquad \text{for all\, } \,t\in I.
\end{equation}

\end{proposition}

\noindent\textit{Proof.} Assume towards a contradiction that for some
$q_{t}^{+}\in\mathrm{\sec}^{+}(t)$ there exists $q\in C\left(  q_{t}^{+},\alpha\right)  \cap K(t)$, $q\neq0$. 
This yields, in view of Proposition~\ref{prop-implies}, that $\mathrm{int\,}K(t)$ is nonempty. Therefore, since $q$ satisfies the open
condition
\[
\langle q_{t}^{+},q\rangle\,>\,\lambda=\cos\alpha,
\]
there is no loss of generality to assume that $q\in\mathrm{int\,}K(t)$.
Therefore, there exist $t_{1}<t_{2}<\ldots<t_{d}<t$ and $\{\mu_{i}\}_{i=1}
^{d}\subset\mathbb{R}_{+}$ such that
\[
u_{i}:=\frac{\gamma(t_{i})-\gamma(t)}{\Vert\gamma(t_{i})-\gamma(t)\Vert}
\quad\text{and}\quad q=\sum\limits_{i=1}^{d}\mu_{i}\,u_{i}.
\]
Fix $\varepsilon>0$ such that $\langle q_{t}^{+},q\rangle>\lambda
+3\varepsilon.$ By continuity, there exists $\delta>0$ such that for all
$s\in(t,t+\delta)$ the vectors
\[
\tilde{u}_{i}:=\frac{\gamma(t_{i})-\gamma(s)}{\Vert\gamma(t_{i})-\gamma
(s)\Vert},\quad i\in\{1,\ldots,d\},
\]
are sufficiently close to $\{u_{i}\}_{i=1}^{d}$ to ensure that
\[
\langle q_{t}^{+},\tilde{q}\rangle>\lambda+2\varepsilon,\qquad\text{where\quad
\ }\tilde{q}=\sum\limits_{i=1}^{n}\mu_{i}\,\tilde{u}_{i}.
\]
Take now $s\in(t,t+\delta)$ in a way that the vector $\hat{q}=\left(
\Vert\gamma(s)-\gamma(t)\Vert\right)  ^{-1}\left(  \gamma(s)-\gamma(t)\right)
$ is sufficiently close to the secant $q_{t}^{+}$ so that $\langle\hat
{q},\tilde{q}\rangle>\lambda+\varepsilon$ or equivalently, $\langle-\hat
{q},\tilde{q}\rangle<-\lambda-\varepsilon.$ Since $\tilde{q},-\hat{q}\in
K(s)\cap \mathbb{S}^{d-1}$, we deduce that $\mathcal{A}(K(s))>\pi-\alpha,$ which
contradicts Corollary~\ref{Cor-implies-} for $s=t$. \hfill$\square$

\bigskip

We shall finally need the following lemma.

\begin{lemma}
\label{Lemma-repulsive} Let $\gamma:I\rightarrow{\mathbb{R}}^{d}$ be a
continuous $\lambda$-curve, with $\lambda\in\lbrack-1,1/d)$ and $\alpha
=\arccos(\lambda).$ Then there exists $\rho>0$ such that for every $t\in I$
and $q_{t}^{+}\in\mathrm{sec}^{+}(t)$, there exists $\xi_{t}\in\mathbb{S}
^{d-1}$ satisfying
\begin{equation}
\langle\xi_{t},u\rangle\leq-\rho<0,\quad\text{for all }u\in K(t) \label{proxi}
\end{equation}
and
\begin{equation}
\langle\xi_{t},q_{t}^{+}\rangle\,\geq\,\rho\,>0\,. \label{proxi0}
\end{equation}

\end{lemma}

\noindent\textit{Proof.} Let $\delta\leq\sqrt{2(1-\lambda)}$ and $\rho
=\delta/2.$ Then for every $t\in I$ and $q_{t}^{+}\in\mathrm{sec}^{+}(t)$, we
have $\mathbb{S}^{d-1}\cap B(q_{t}^{+},\delta)\subset C(q_{t}^{+},\alpha)$. We deduce
from Proposition~\ref{prop-gap} that the $\delta$-enlargement of the cone
$K(t)$ satisfies:
\[
K(t)_{\delta}\cap\mathrm{sec}^{+}(t)=\emptyset.
\]
Setting $\tilde{N}(t)=(K(t)_{\delta})^{o}$ and  $N(t)=K(t)^{o}$ (the polar of
$K(t)_{\delta}$ and $K(t)$ respectively), we deduce by \eqref{polar} that
\begin{equation}
\bar{B}(\xi,\delta)\cap \mathbb{S}^{d-1}\subset N(t)\text{,\quad for every }\xi\in\tilde
{N}(t)\cap \mathbb{S}^{d-1}. \label{r1}
\end{equation}
Let us now fix $q_{t}^{+}\in\mathrm{sec}^{+}(t).$ Then by the bipolar theorem
we get $q_{t}^{+}\notin\tilde{N}(t)^{o}=K(t)_{\delta},$ that is, there exists
$\tilde{\xi}\in\tilde{N}(t)\cap \mathbb{S}^{d-1}$ such that $\langle\tilde{\xi},q_{t}
^{+}\rangle>0.$ Maximizing the functional $q_{t}^{+}$ over the closed ball $\bar{B}(\tilde{\xi},\rho)$ 
we obtain $\xi_{t}\in\mathbb{S}^{d-1}$ such that (\ref{proxi0}) holds. Since $B(\xi
_{t},\rho)\subset B(\tilde{\xi},\delta)\subset N(t),$ we easily deduce that
(\ref{proxi}) also holds. \hfill$\square$

\bigskip

\begin{figure}[h!]
\includegraphics[width=0.6\textwidth]{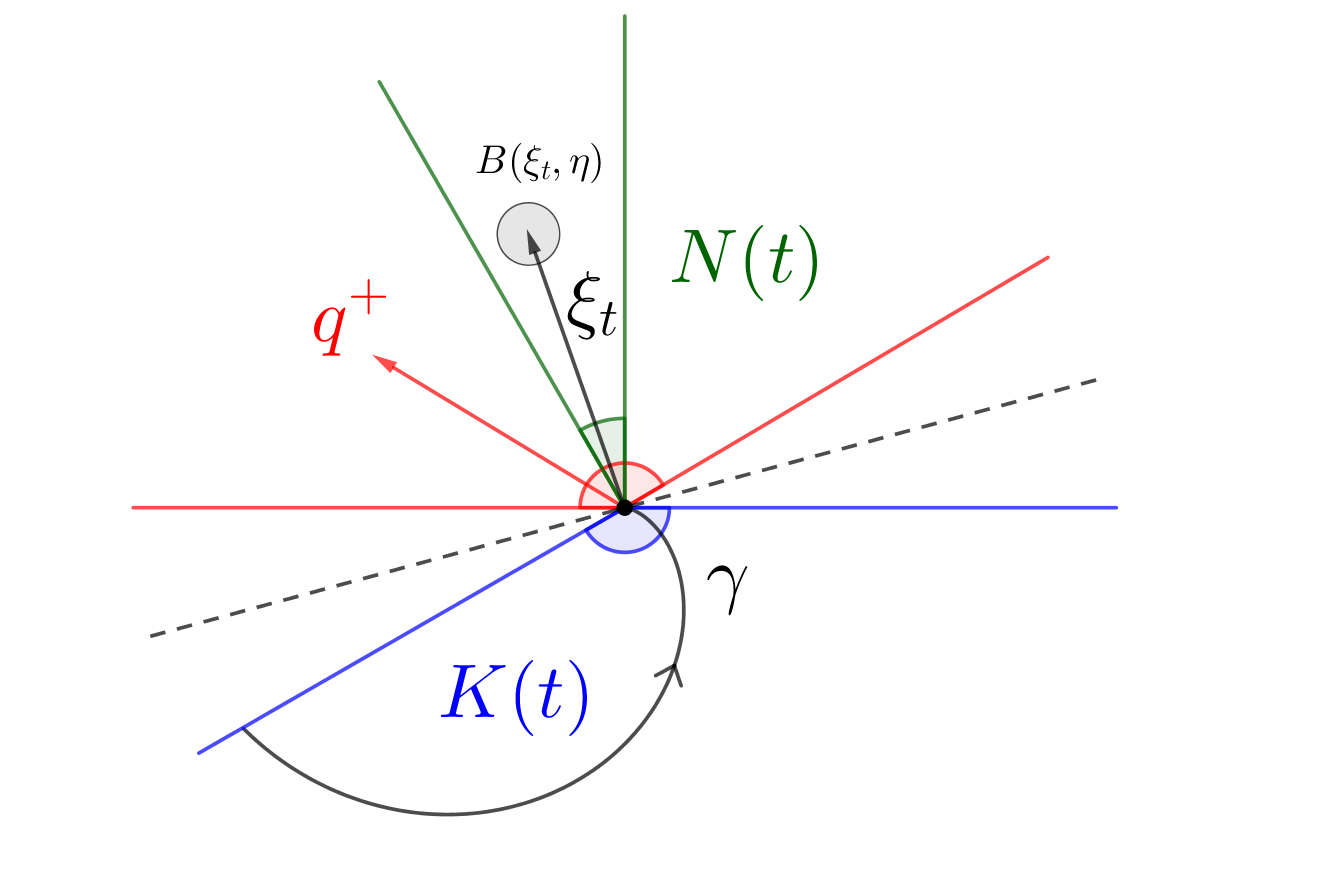}
\caption{The initial part of the curve generates the cone $K(t)$ (in blue) with aperture
$\mathcal{A}(K(t))\leq\pi-\alpha$ and avoids the cone generated by the positive secants (in red).}
\label{circletangent}
\end{figure}

We are now ready to prove the main result of this section.

\begin{theorem}
[rectifiability]\label{Thm_l-curve-rect}Every continuous $\lambda$-curve
$\gamma:I\rightarrow{\mathbb{R}}^{d}$ with $\lambda<1/d$ is rectifiable. In
particular, bounded $\lambda$-curves with $\lambda<1/d$ have finite length.
\end{theorem}

\noindent\textit{Proof.} We may assume that $I=[0,+\infty)$ and that
$\gamma\ $is bounded. Set $\eta=\rho/3,$ where $\rho$ is given by Lemma
\ref{Lemma-repulsive}. Since $\mathbb{S}^{d-1}$ is compact, there exists an
$\eta$-net $\mathcal{F}:=\{\xi_{1},\cdots,\xi_{N}\}$, satisfying that for
every $v\in\mathbb{S}^{d-1}$, there exists $i\in\{1,\cdots,N\}$ such that
$\langle v,\xi_{i}\rangle>\eta$ (that is, $v$ is $\eta$-close to some $\xi
_{i}\in\mathcal{F}$). Then we deduce from Lemma~\ref{Lemma-repulsive} that for
every $t\in I$ and $q_{t}^{+}\in{\sec}^{+}(t)$, there exists $\xi_{i}
\in\mathcal{F}$ such that
\begin{equation}
\langle\xi_{i},q^{+}\rangle>2\eta\qquad\text{and\qquad}\langle\xi_{i}
,u\rangle\leq-2\eta<0,\quad\text{for all }u\in K(t).\label{eq:sec1}
\end{equation}
Reasoning by contradiction we can prove the existence of some $\delta_{t}>0$
such that for every $s\in\lbrack t,t+\delta_{t})$ there exists $q_{t,s}^{+}
\in\mathrm{sec}^{+}(t)$ such that
\begin{equation}
\left\Vert \frac{\gamma(s)-\gamma(t)}{||\gamma(s)-\gamma(t)||}-q_{t,s}
^{+}\right\Vert <\eta.\label{eq:sec2}
\end{equation}
Combining the above we deduce that for every $t\in I$ and $s\in\lbrack
t,t+\delta_{t})$, there exists $\xi_{i}\in\mathcal{F}$ such that
\begin{equation}
\langle\xi_{i},\gamma(s)-\gamma(t)\rangle\geq\,\eta\,\Vert\gamma
(s)-\gamma(t)\Vert.\label{f-1}
\end{equation}
On the other hand, it follows directly from (\ref{eq:sec1}) that for every
$\tau\in\lbrack0,t)$
\begin{equation}
\langle\xi_{i},\gamma(t)-\gamma(\tau)\rangle\geq\,\eta\,\Vert\gamma
(t)-\gamma(\tau)\Vert.\label{f-2}
\end{equation}
Considering for $i\in\{1,\ldots,N\}$ the projection operator
\[
\left\{
\begin{array}
[c]{l}
\pi_{i}:\mathbb{R}^{d}\rightarrow\mathbb{R\xi}_{i} \medskip \\
\pi_{i}(x)=\,\langle\mathbb{\xi}_{i},x\rangle\,\xi_{i}
\end{array}
\right.
\]
we define $W_{i}(t)$ to be the width of the projection of the initial part of
the curve $\Gamma(t)$ onto $\mathbb{R\xi}_{i}$, that is,
\[
W_{i}(t):=\mathcal{H}^{1}(\pi_{i}(\Gamma(t))),\quad t\in I,
\]
where $\mathcal{H}^{1}$ denotes the 1-dimensional Lebesgue measure. Notice
that $\mathcal{H}^{1}(\pi_{i}(\Gamma(t)))$ is simply the length of the bounded
interval $\pi_{i}(\Gamma(t))$ of $\mathbb{R\xi}_{i}$. It follows readily that
for every $i\in\{1,\ldots,N\}$ the function $t\mapsto W_{i}(\tau)$ is
non-decreasing on $[0,T_{\infty})$ and bounded above by $r:=\mathrm{diam}
(\gamma(I))$. Therefore, the function
\[
W_{\mathcal{F}}(t):=\sum_{i=1}^{N}W_{i}(t),
\]
is non-decreasing on $I$ and bounded above by $Nr.$ We now deduce from
(\ref{f-1}) and (\ref{f-2}) that for every $t\in I$ there exists $\delta
_{t}>0$ such that for all $s\in\lbrack t,t+\delta_{t})$ we have
\begin{equation}
W_{\mathcal{F}}(s)-W_{\mathcal{F}}(t)\geq\eta\,\Vert\gamma(s)-\gamma
(t)\Vert.\label{aa}
\end{equation}
The result follows via a standard argument if we establish that for any
$a,b\in I$ with $a<b$ it holds:
\begin{equation}
W_{\mathcal{F}}(b)-W_{\mathcal{F}}(a)\geq\eta\,\Vert\gamma(b)-\gamma
(a)\Vert.\label{ab}
\end{equation}
Let us assume, towards a contradiction, that (\ref{ab}) does not hold, that
is,
\[
W_{\mathcal{F}}(b)-W_{\mathcal{F}}(a)+\varepsilon<\eta\,\Vert\gamma
(b)-\gamma(a)\Vert\text{,}\quad\text{ for some }\varepsilon>0.
\]
Set $\sigma(t)=\sup\{s>t:\; \eqref{aa} \,\,\text{holds}\}$, for $t\in\lbrack a,b)$. Then
our assumption yields that for every $t\in\lbrack a,b)$ we have $a\leq
t+\delta_{t}\leq\sigma(t)<b.$ Using transfinite induction we construct a
(necessarily) countable set $\Lambda=\{t_{\mu}\}_{\mu\leq\hat{\varsigma}}$ by
setting $t_{1}=a,$ $t_{\mu}=\sigma(t_{\mu^{-}})$ if $\mu=\mu^{-}+1$ is a
successor ordinal, and $t_{\mu}=\sup\{t_{\nu}:\nu<\mu\}$ if $\mu$ is a limit
ordinal and we stop when $t_{\hat{\varsigma}}=b$. Let now $\{\varepsilon
_{n}\}_{n\in\mathbb{N}}\subset(0,\varepsilon)$ with 
$\sum_{n\in N}\varepsilon_{n}=\varepsilon\eta^{-1}.$ Let $i:\Lambda\rightarrow\mathbb{N}$ be
an injection of $\Lambda$ into $\mathbb{N}$. Then denoting by $\mu^{+}$ the
successor of $\mu,$ we obtain by continuity, that for each ordinal $\mu$ there
exists $t_{\mu}\leq s_{\mu}<\sigma(t_{\mu}):=t_{\mu^{+}}$ such that
$||\gamma(s_{\mu})-\gamma(t_{\mu^{+}})\Vert<\varepsilon_{i(\mu)}$. We deduce
by (\ref{aa}):
\begin{align*}
\Vert\gamma(b)-\gamma(a)\Vert & \leq
\sum_{\mu\in\Lambda}
\Vert\gamma(t_{\mu^{+}})-\gamma(t_{\mu})\Vert\leq
\sum_{\mu\in\Lambda}
\left(  \Vert\gamma(s_{\mu})-\gamma(t_{\mu})\Vert+\varepsilon_{i(\mu)}\right)
\\
& \leq\frac{1}{\eta}\left(
\sum_{\mu\in\Lambda}
\left(  W_{\mathcal{F}}(s_{\mu})-W_{\mathcal{F}}(t_{\mu})\right)
\,+\,\varepsilon\right)  \leq\frac{1}{\eta}\left(  W_{\mathcal{F}
}(b)-W_{\mathcal{F}}(a)+\varepsilon\right)  ,
\end{align*}
which contradicts (\ref{ab}).\hfill$\square$

\bigskip

\begin{remark}
[universal constant] The above proof reveals that the length $\ell(\gamma)$ of any
$\lambda$-curve lying in a set of diameter $r$ is bounded by the
quantity $N\cdot\eta^{-1}\cdot r$. Since the constant $\eta>0$ is
determined in Lemma~\ref{Lemma-repulsive}, it only
depends on $\lambda$ and the dimension $d$ of the space (in
particular, it is independent of the specific $\lambda$-curve $\gamma$). Since
$N$ (the cardinality of the net $\mathcal{F}$) also depends exclusively on $\eta$ and the
dimension $d$, we conclude that for a given $\lambda\in\lbrack-1,1/d)$ there exists
a prior bound for the lengths of all $\lambda$-curves $\gamma$ lying inside a
prescribed bounded subset of $\mathbb{R}^{d}$.
\end{remark}

\begin{remark}[Double cone property]\label{essence}
A close inspection of Theorem~\ref{Thm_l-curve-rect} shows that the proof depends exclusively on \eqref{proxi}--\eqref{proxi0} which in turn depend on \eqref{todo}. Therefore, every bounded continuous curve $\gamma$ satisfying \eqref{todo} has finite length.
\end{remark} 

\bigskip

%%%%%%%%%%%%%%%
%% Section 3

\section{A bounded curve with the $\lambda$-cone property and infinite length}

In this section we consider continuous right differentiable curves
$\gamma:I\rightarrow\mathbb{R}^{d}$ satisfying the $\lambda$-cone property
(Definition~\ref{def_l-cone-prop}). In the sequel we denote by $\gamma
^{\prime}(\tau)$ the right derivative of $\gamma$ at the point $\tau$ and we
assume this derivative is nonzero. Observe that in this case we have
\[
\sec^{+}(t)=\left\{  \frac{\gamma'(t)}{\Vert\gamma'(t)\Vert
}\right\}  .
\]
So $\gamma$ satisfies the $\lambda$-cone property
if,  for all $t,\tau\in I$ with $t<\tau$, \eqref{sec-eel} holds, or equivalently:
\[
\langle\gamma'(\tau),\gamma(t)-\gamma(\tau)\rangle\leq\lambda
\,\vert\vert\gamma'(\tau)\vert\vert \,
\vert\vert \gamma(t)-\gamma(\tau)\vert\vert .
\]
This means that the angle between the vectors $\gamma'(\tau)$ and 
$\gamma(t)-\gamma(\tau)$ is greater or equal to $\alpha$, where $\alpha=\arccos(\lambda)$.
%Recalling \eqref{cone-notation} and \eqref{a(l)}, 
We simplify the notation by setting
\begin{equation}
C(t,\alpha):=\gamma(t)+C\left(  \frac{\gamma'(t)}{\Vert\gamma'(t)\Vert},\alpha\right)  . \label{simple}
\end{equation}
A curve $\gamma$ satisfying the above property will be also called a $\lambda
$-eel. The reason is as follows: the set $\Gamma(\tau):=\{\gamma
(t);\,t\in I,\,t<\tau\}$ is the apparent body (or tail) of a $\lambda$-eel at
time $\tau$ going out of a hole. The cone $C(\tau,\alpha)$ represents what the
$\lambda$-eel can see at time $\tau$. The $\lambda$-cone property just says
that the $\lambda$-eel never sees its apparent tail. Notice that $\pi/2$-eels
correspond to self-expanded curves. Therefore, if the range of $\gamma$ is
bounded and $\gamma$ is a $\pi/2$-eel, then its length is finite (\cite{DDDL},
\cite{LMV}).\smallskip

%%%%%%%%%%%
%%%%%%%%%%%%

Recall from the introduction that a curve $\gamma$ is
self-expanded if for all $\tau\in I$, the map $t\mapsto
d(\gamma(t),\gamma(\tau))$ is non decreasing on $I\cap\lbrack\tau,+\infty)$. The following lemma illustrates that one can also associate a Lyapunov function to  $\lambda$-eels.

\begin{lemma}\label{lem1}
If $\gamma:I\rightarrow {\mathbb{R}}^{d}$ is a $\lambda-$eel, then the function 
\begin{equation*}
t\mapsto \|\gamma(t_1)-\gamma(t)\|+\lambda\ell(\gamma_{|[t_1,t]})
\end{equation*}
is non-decreasing on $I\cap[t_1,\infty).$
\end{lemma}

\begin{proof}
By definition,
\[
\frac{d}{d\tau}(\|\gamma(\tau)-\gamma(t)\|)=\left\langle\gamma'(\tau),\frac{\gamma(\tau)-\gamma(t)}{\|\gamma(\tau)-\gamma(t)\|}\right\rangle\geq -\lambda \|\gamma'(\tau)\|\quad\forall\, t<\tau.
\] 
For $t<t_1<t_2$, integrating for $\tau\in [t_2, t_3]$ we obtain
\[
\int_{t_2}^{t_3}\frac{d}{d\tau}(\|\gamma(\tau)-\gamma(t)\|)d\tau\geq -\lambda \int_{t_2}^{t_3}\|\gamma'(s)\|ds\quad\forall\, t<\tau,
\] 
which implies
\[
\|\gamma(t_3)-\gamma(t)\|-\|\gamma(t_2)-\gamma(t)\|\geq -\lambda \ell(\gamma_{|[t_2,t_3]}).
\] 
Since $ \ell(\gamma_{|[t_2,t_3]})= \ell(\gamma_{|[t_1,t_3]})- \ell(\gamma_{|[t_1,t_2]})$ the conclusion follows.
\end{proof}

%%%%%%
%%%%%%
Our main aim now is to prove the following result.

\begin{theorem}
[$\lambda$-eel of infinite length]\label{main} Assume 
$\lambda={\frac{1}{\sqrt{5}}}$ 
(i.e. $\alpha=\arccos{\frac{1}{\sqrt{5}}}$), and let $B=\overline{B}(0,1)$
the unit ball of ${\mathbb{R}}^{3}$. Then, there
exists a $\lambda$-eel $\gamma:[0,+\infty)\to B$ of infinite length.
Moreover $\lim\limits_{t\to\infty}\gamma(t)$ exists.
\end{theorem}

The proof of Theorem~\ref{main} is constructive: the construction will be
carried out in three steps organized in subsections. Let us mention that the
result remains true if we require $\gamma$ to be $\mathcal{C}^{1}$-smooth (and
probably even $\mathcal{C}^{\infty}$-smooth), but the construction would then
become less transparent. Before we proceed, let us make the following remark.

\begin{remark}
Let us denote by $\lambda_{\ast}$ the infimum of all $\lambda$ for which there
exists a bounded $\lambda$-eel of infinite length inside the unit ball of
${\mathbb{R}}^{3}$. Since for $\lambda=0$ we obtain a self-expanded curve, it
follows from the above theorem that $0\leq\lambda_{\ast}\leq\frac{1}{\sqrt{5}}$. Notice that we
cannot readily conclude that $\lambda_{\ast}$ is strictly greater than $0$. (Nonetheless, according to \cite{MP} or \cite{DDDL}, for $\lambda=0$ bounded $\lambda$-eels have finite length.)
 \end{remark}

%\section{Constructing $\lambda$-eels of infinite length}
\medskip

\subsection{Helicoidal maps%Constructing eels that avoid a thin cylinder
}

Let us start by constructing a helicoidal curve along the $z$-axis,
which is self-expanded.

\begin{lemma}
\label{selfexpanded} There exists a positive constant $\mu<1/2$ such that, if
$\gamma:{\mathbb{R}}\rightarrow{\mathbb{R}}^{3}$ is a spiral of the form
\begin{equation}
\gamma(t)=\left(  r\,\cos t,\,r\,\sin t,\,\mu rt\right)  ,\quad t\in
\mathbb{R}\text{,}\label{g-31}
\end{equation}
then $\gamma$ is self-expanded $($hence $\gamma$ satisfies the $\lambda$-cone
property for all $\lambda\in\lbrack0,1))$.
\end{lemma}

\noindent\textit{Proof.} Let $\gamma:{\mathbb{R}}\rightarrow{\mathbb{R}}^{3}$
be a spiral along a cylinder of radius $r>0$ of the form (\ref{g-31}) and let
us show that $\gamma$ is a self-expanded curve. By symmetry, this amounts to verify
that
\[
a(t):=\langle\gamma'(0),\gamma(t)-\gamma(0)\rangle\leq0,\quad\text{for
all }\,t<0.
\]
We check easily that $\gamma(0)=\left(  r,0,0\right)  $ and $\dot{\gamma
}(0)=\left(  0,r,\mu r\right)  $, so that $a(t)=r^{2}(\sin t+\mu^{2}t).$ Since
$\sup\,\{-t^{-1}\sin t:$ $t<0\}<1/4$, we deduce that there exists $\mu <\frac{1}{2}$ 
such that the curve $\gamma$ is self-expanded. \hfill$\square$

\bigskip
\noindent
\bf Notation. \rm
Throughout this subsection, $\gamma$ will refer to the curve given in
Lemma \ref{selfexpanded} and $\mu<1/2$ will be the constant fixed there.\smallskip

The following lemma says that the curve $\gamma$ constructed in the previous lemma  
satisfies that for each $\tau$, the associated cone
$C(t,\alpha)$, $\alpha=\arccos(1/\sqrt{5})$ does not meet the $z$-axis, that is,
the axis of evolution of the spiral curve.

\begin{lemma}
Let $\gamma:{\mathbb{R}}\rightarrow{\mathbb{R}}^{3}$ be a spiral of the
form (\ref{g-31}). If $\lambda={1/\sqrt{5}}$ and $\alpha=\arccos(\lambda)$, 
then the cone $C(t,\alpha)$ does not intersect the line
parametrized by $\ell(z)=(0,0,z)$.
\end{lemma}

\noindent\textit{Proof.} Under the notation of the previous lemma, it is
enough to verify that for all $z\in{\mathbb{R}}$
\[
\langle\gamma'(0),\ell(z)-\gamma(0)\rangle\leq\frac{1}{\sqrt{5}}\,\Vert
\dot{\gamma}(0)\Vert\,\left\Vert \ell(z)-\gamma(0)\right\Vert .
\]
The above condition reads
\begin{equation}
\mu rz\leq\frac{1}{\sqrt{5}}\,\sqrt{r^{2}(1+\mu^{2})}\,\sqrt{r^{2}+z^{2}},\qquad\text{for
all }z\in{\mathbb{R}},\label{eq1}
\end{equation}
or equivalently,
\begin{equation}
\frac{(z/r)}{\sqrt{1+(z/r)^{2}}}\leq\,\frac{\sqrt{(1+\mu^{2})}}{\mu\sqrt{5}
},\qquad\text{for all }z\in{\mathbb{R}}.\label{eq2}
\end{equation}
Since $t\mapsto t^{-1}\sqrt{1+t^{2}}$ is decreasing for $t>0$ and $\mu<1/2$,
we have $\mu^{-1}\sqrt{1+\mu^{2}}>\sqrt{5}$, therefore \eqref{eq2} is
satisfied.\hfill$\square$

\bigskip

We shall now enhance in the above construction to deduce that
the cone $C(\tau,\alpha)$ avoids a thin (infinite)
cylinder%
\[
\mathrm{Cyl\,}(r_{0})=\{(x,y,z)\in{\mathbb{R}}^{3};\,x^{2}+y^{2}=r_{0}
^{2},\,z\in\mathbb{R}\}
\]
containing the $z$-axis. Indeed, taking $r_{0}<<r$ the above cylinder
is very close to the $z$-axis, therefore we obtain (almost) the same result as
before. This is formulated in the next lemma.

\begin{lemma}
\label{smallcylinder} There exists an integer $N\geq2$ such that  
whenever $r=Nr_{0}$ and $\alpha=\arccos{1/\sqrt{5}}$, we
have:
\[
C(\alpha,\tau)\cap\mathrm{Cyl\,}(r_{0})=\emptyset\text{,\quad for all }
\tau\geq0.
\]

\end{lemma}

\noindent\textit{Proof.} We consider again the curve $\gamma$ given by
(\ref{g-31}). Thanks to the symmetry, it is enough to check the assertion for
$\tau=0$. Therefore, for $\sigma(\theta,z)=\left(  r_{0}\cos\theta,r_{0}
\sin\theta,z\right)$, it is enough to verify
\[
\langle\gamma'(0),\sigma(\theta,z)-\gamma(0)\rangle\leq\cos
\alpha\,\left\Vert\gamma'(0)\right\Vert \,\left\Vert \sigma
(\theta,z)-\gamma(0)\right\Vert \qquad\forall\theta\in\lbrack0,2\pi],\,\forall
z\geq0,
\]
where $\gamma(0)=\left(  r,0,0\right)  \ $and $\gamma'(0)=\left(  0,r,\mu
r\right)  .$The above condition reads
\[
rr_{0}\sin\theta+\mu rz\leq\cos\alpha\,\sqrt{r^{2}(1+\mu^{2})}\,\sqrt{\left(
r_{0}\cos\theta-r\right)  ^{2}+r_{0}^{2}\sin^{2}\theta+z^{2}}\quad
\forall\theta\in\lbrack0,2\pi],\,\forall z\geq0.
\]
Dividing by $rr_{0}$, setting $w=z/r_{0}$, and since $\cos\alpha={1/\sqrt{5}}$, we
deduce
\[
\sin\theta+\mu w\leq\frac{1}{\sqrt{5}}\,\sqrt{1+\mu^{2}}\,\sqrt{\left(  \frac{r}{r_{0}
}-1\right)  ^{2}+2\frac{r}{r_{0}}(1-\cos\theta)+w^{2}}.
\]
Setting $r=Nr_{0}$ we obtain the condition
\[
\frac{1}{\sqrt{5}}\geq\frac{1}{\sqrt{1+\mu^{2}}}\,\sup_{\theta\in\lbrack0,2\pi
],w\in\mathbb{R}}\left\{  \frac{\sin\theta+\mu|w|}{\sqrt{(N-1)^{2}
+2N(1-\cos\theta)+w^{2}}}\right\}  .
\]
But for any $\theta\in\lbrack0,2\pi]$, $u=|w|\geq0$
\[
\frac{\sin\theta+\mu u}{\sqrt{(N-1)^{2}+2N(1-\cos\theta)+u^{2}}}\leq
\frac{1+\mu u}{\sqrt{(N-1)^{2}+u^{2}}}
\]
and
\[
\sup_{u\geq0}\left\{  \frac{1+\mu u}{\sqrt{(N-1)^{2}+u^{2}}}\right\}
=\frac{1}{N-1}\,\sqrt{1+\mu^{2}(N-1)^{2}}\longrightarrow\mu\,\,\,\text{ as
}\,\,\,N\rightarrow+\infty.
\]
Since  $\mu\left(  1+\mu^{2}\right)  ^{-1/2}<\left(  \sqrt{5}\right)  ^{-1}$,
we can choose $N$ large enough such that
\[
\frac{\sqrt{1+\mu^{2}(N-1)^{2}}}{(N-1)\sqrt{1+\mu^{2}}}<\frac{1}{\sqrt{5}}.
\]
Therefore, for this choice of $N$, we get
$C(0,\alpha)\cap\mathrm{Cyl\,}(r_{0})=\emptyset$. \hfill$\square$

\bigskip

Let $\gamma$ be given by (\ref{g-31}). We shall now include a further
restriction. We shall show that the cone $C(\tau,\alpha)$ associated to $\gamma$ also
avoids radial segments $S$ of the form:
\[
S=\{(x,0,0);\,0\leq x\leq r\}.
\]
This is the aim of the following lemma.

\begin{lemma}
\label{forward} If $\lambda=\frac{1}{\sqrt{5}}$ and $\alpha=\arccos(\lambda),$
then $C(\tau,\alpha)\cap S=\emptyset$ for all $\tau\geq0$.
\end{lemma}

\noindent\textit{Proof.} It is enough to verify
\[
\langle\gamma'(\tau),(x,0,0)-\gamma(\tau)\rangle\leq\cos\alpha
\,\left\Vert \dot{\gamma}(\tau)\right\Vert \,\left\Vert (x,0,0)-\gamma
(\tau)\right\Vert ,\qquad\text{for all}\,\,0\leq x\leq r,
\]
where
\[
\gamma(\tau)=\left(  r\,\cos\tau,\,r\,\sin\tau,\,\mu r\tau\right)
\quad\mbox{and}\quad\dot{\gamma}(\tau)=\left(  -r\sin\tau,\,r\cos\tau,\,\mu
r\right)  .
\]
Setting $\lambda=\cos\alpha$ and simplifying by $r$, we obtain for all $0\leq
x\leq r$
\begin{equation}
-x\sin\tau-\mu^{2}r\tau\,\leq\,\lambda\,\sqrt{1+\mu^{2}}\,\sqrt{(x-r\cos
\tau)^{2}+r^{2}\sin^{2}\tau+\mu^{2}r^{2}\tau^{2}},\label{dev}
\end{equation}

Notice that $\mu$ satisfies $\sin\tau+\mu^{2}\tau>0$ for every $\tau\geq0$. Therefore,
\[
-x\sin\tau-\mu^{2}r\tau\leq-x(\sin\tau+\mu^{2}\tau)\leq0,
\]
so (\ref{dev}) is clearly satisfied.\hfill$\square$

\bigskip

\subsection{Arbitrary long eels inside a bounded cylinder}

We are now ready to construct arbitrarily long $\lambda$-eels lying inside
the following bounded cylinder:
\begin{equation}
\mathrm{Cyl\,}(r,[a,a+2\pi\mu r]):=\{(x,y,z)\in{\mathbb{R}}^{3};\,x^{2}
+y^{2}=r^{2},\,a\leq z\leq a+2\pi\mu r\}.\label{b-cyl}
\end{equation}
Indeed we have the following result.

\begin{proposition}
\label{ge1} Let $\lambda\geq1/\sqrt{5}$ and
let $\mathrm{Cyl\,}(r,[a,a+2\pi\mu r])$ be the bounded cylinder defined in
(\ref{b-cyl}). Then there exists a $\lambda$-eel
\[
\gamma:I\longmapsto\mathrm{\mathrm{Cyl\,}}(r,[a,a+2\pi\mu r])
\]
whose length is greater than $1$. Moreover, the initial point of $\gamma
\ $lies in the upper part of the cylinder ($z=a+2\pi\mu r$) while the last
point lies at the bottom ($z=a$).
\end{proposition}

\noindent\textit{Proof.} Without loss of generality, we assume $a=0$.
Below, $N$ is a fixed integer given by Lemma~\ref{smallcylinder}. Let us
fix an odd integer $n$ such that $2\pi\mu rn>1$. Then for $1\leq k\leq n$, we
define internal cylinders
\[
C_{k}:=\mathrm{\mathrm{Cyl\,}}(\frac{r}{N^{n-k}},[0,2\pi\mu r])=\{(x,y,z)\in
{\mathbb{R}}^{3};\,x^{2}+y^{2}=\left(  \frac{r}{N^{n-k}}\right)  ^{2},\,0\leq
z\leq2\pi\mu r\}.
\]
For $k=2\ell+1\leq n$ (odd) we define a downward spiral curve $\gamma
_{k}^{\downarrow}$ as follows:
\[
\gamma_{k}^{\downarrow}(t)=\frac{r}{N^{n-k}}\bigl(
\cos(t),\,  \sin(t),\,\mu(2\pi N^{n-k}-t)\bigr),\qquad\text{ for }0\leq t\leq2\pi
N^{n-k}.
\]
while for $k=2\ell\leq n$ (even) we define an upward spiral curve $\gamma
_{k}^{\uparrow}$ as follows:
\[
\gamma_{k}^{\uparrow}(t)=\frac{r}{N^{n-k}}\bigl(
\cos(t),\,  \sin(t),\,\mu t\bigr),\qquad\text{for }0\leq t\leq2\pi N^{n-k}.
\]
Notice that if $k$ odd,
\[
\gamma_{k}^{\downarrow}(0)=(\frac{r}{N^{n-k}},0,2\pi\mu r)\quad\text{and}
\quad\gamma_{k}^{\downarrow}(2\pi N^{n-k})=(\frac{r}{N^{n-k}},0,0),
\]
while for $k$ even
\[
\gamma_{k}^{\uparrow}(0)=(\frac{r}{N^{n-k}},0,0)\quad\text{and}\quad\gamma
_{k}^{\uparrow}(2\pi N^{n-k})=(\frac{r}{N^{n-k}},0,2\pi\mu r).
\]
Each spiral $\gamma_{k}$ lies on the surface of the cylinder $C_{k}$ and makes
$N^{n-k}$ loops to reach the upper part of the cylinder starting from the
bottom and going upwards if $k$ is even (respectively, to reach the bottom,
starting from the upper part and going downward, if $k$ is odd). We finally
define parametrized segments $e_{k}^{+}$ joining the end point of $\gamma_{k}^{\downarrow}$ 
to the initial point of $\gamma_{k}^{\uparrow}$ (for $k=2\ell+1$), and respectively $e_{k}^{-}$ 
joining the end point of $\gamma_{k}^{\uparrow}$ to the initial point of $\gamma_{k+1}^{\downarrow}$ (for
$k=2\ell$), that is:
\[
e_{k}^{+}(t)=\left(\frac{r}{N^{n-k}}(1+t(N-1)),0,2\pi\mu r\right)
\,\text{ and }\,\text{ }e_{k}^{-}(t)=\left(\frac{r}{N^{n-k}
}(1+t(N-1)),0,0\right)  ,\quad t\in\lbrack0,1].
\]

\begin{figure}[h!]
  \includegraphics[width=0.6\textwidth]{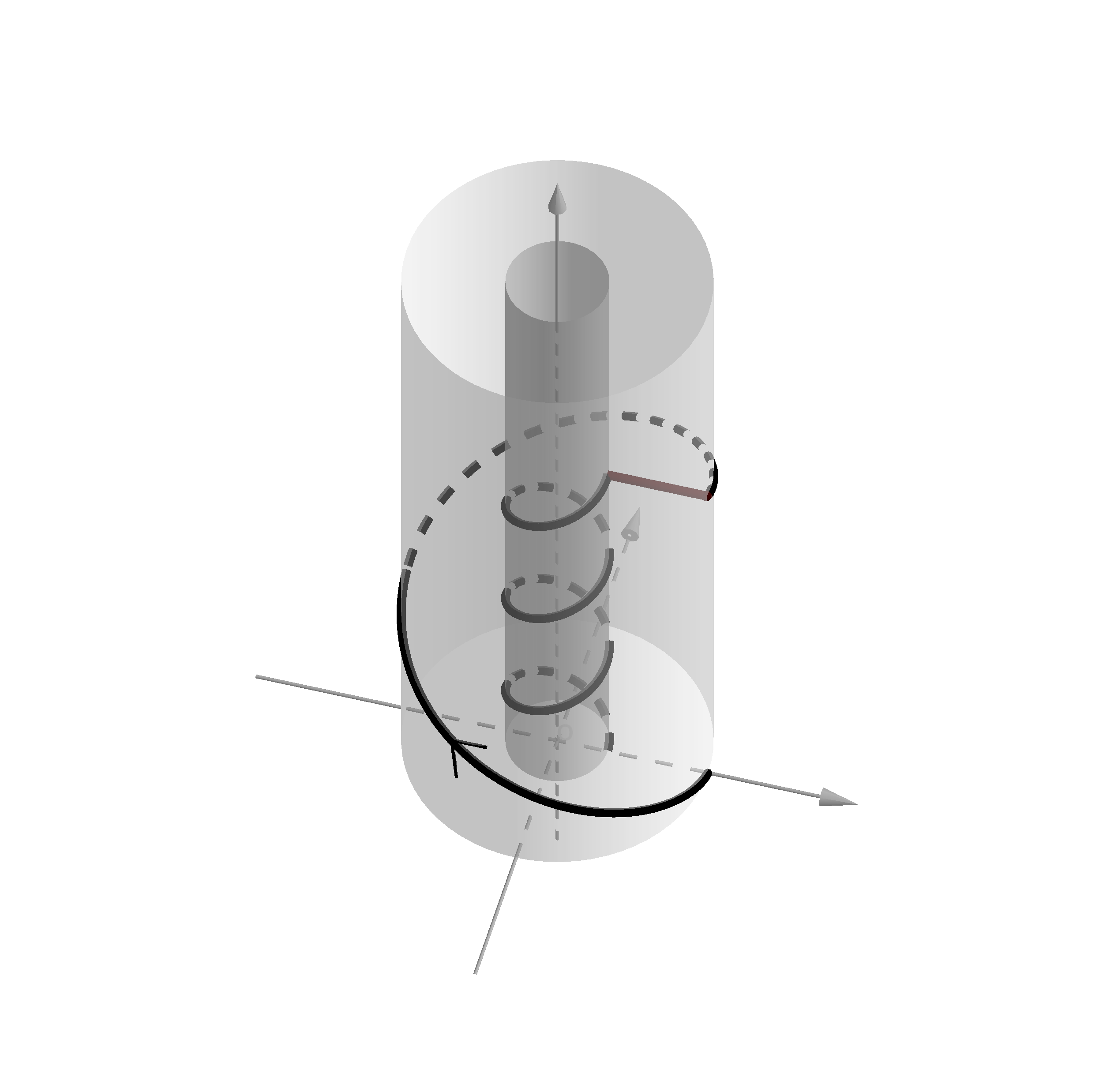}
 \caption{A block of the construction}
 \label{circletangent}
 \end{figure}

The curve $\gamma$ will now be defined concatenating the above curves: we
start with $k=1$ and the downward spiral $\gamma_{1}^{\downarrow}$ and we
concatenate with the segment $e_{1}^{-}$. We continue with the upward spiral
$\gamma_{1}^{\uparrow}$ and the segment $e_{1}^{+}$ and concatenate with
$\gamma_{2}^{\downarrow}$ ($k=2$), then the segment $e_{2}^{-}$ and so on, up
to the final downward spiral $\gamma_{n}^{\downarrow}$. The resulting curve is
clearly continuous. Applying Lemma~\ref{smallcylinder} and Lemma~\ref{forward}
we deduce that $\gamma$ is a $\lambda$-eel. The length of $\gamma$ is clearly greater than~$1$
since we cross $n$ times the cylinder of length $2\pi\mu r$ (and we have taken $n\ge2$ such that $2\pi\mu rn>1$). \hfill$\square$

\bigskip

\begin{remark}
\label{type} Proposition~\ref{ge1} ensures, by rescaling, that we can construct
arbitrarily long $\lambda$-eels inside arbitrarily small cylinders. The
$\lambda$-eel $\gamma$ inside the cylinder
$\mathrm{\mathrm{Cyl\,}}(r,[a,b])$, where $b=a+2\pi\mu r,$ 
is obtained by concatenating pieces of
three different types:

\smallskip\noindent- Type 1: a spiral going downward: $\gamma_i^{\downarrow
}(t)=(\rho\cos(t),\rho\sin(t),b-\rho\mu t)$,

\smallskip\noindent- Type 2: a spiral going upward: $\gamma_i^{\uparrow
}(t)=(\rho\cos(t),\rho\sin(t),a+\rho\mu t)$,

\smallskip\noindent- Type 3: a segment parametrized by $e_i^{-}(t)=(t,0,a)$ or
by $e_i^{+}(t)=(t,0,b)$.
\end{remark}

\begin{remark}
It is possible to modify slightly the above construction to get a $\lambda$-ell (with $\lambda=1/\sqrt{5}$)
$\gamma:(-\infty,0]\to\mathrm{Cyl\,}(r,[a,a+2\pi\mu r])$, with infinite length,
and such that $\lim\limits_{t\to -\infty}\gamma(t)$ does not exist.
Since the curve constructed $\gamma$ above depends on the parameter $n$,
let us denote it $\gamma_n$. We can assume, without loss of generality, by choosing a suitable parametrization, 
that $\gamma$ is defined on $[-n,0]$ and that for each $n$,
the restriction of $\gamma_{n+1}$ to  $[-n,0]$ coincides with $\gamma_n$.
Now we define $\gamma$ on $(-\infty,0]$, satisfying, for each $n$, $\gamma_{\vert[-n,0]}=\gamma_n$.
Since each $\gamma_n$ is a $\lambda$-ell, it is clear that $\gamma$ is  a $\lambda$-ell.
Morover, the $z$-coordinate of $\gamma(t)$ oscillates infinitely many times between $a$ and $a+2\pi\mu r$.
This shows both that $\gamma$ has  infinite length
and that $\lim\limits_{t\to -\infty}\gamma(t)$ does not exist.
\end{remark}

\subsection{Constructing bounded eels of infinite length in 3D}

To construct a bounded $\lambda$-eel with infinite length, we need to glue
together curves of length greater than $1$ (constructed in the previous
subsection) that lie each time in prescribed disjoint bounded cylinders, all
taken along the $z$-axis, of the form $C_{n}:=\mathrm{Cyl\,}(r_{n}
,[a_{n},b_{n}])$ with $a_{n}>b_{n+1}$ and $r_{n}\searrow0^{+}$. To construct
efficiently such a curve, and to establish that it is a $\lambda$-eel, we
shall need the following result, asserting that a $\lambda$-eel lying in a
small cylinder does not see a bigger remote cylinder of the same axis.

\begin{lemma}
\label{bigcylinder} Let $\lambda=1/\sqrt{5}$, $\alpha=\arccos(\lambda),$
and let us set
\[
\mathrm{Cyl\,}(R,[a,b]):=\{(x,y,z)\in{\mathbb{R}}^{3};\,x^{2}+y^{2}\leq
R,\,a\leq z\leq b\}.
\]
Then there exists $M>1$ such that, for every $r\in(0,R/2)$ and $a^{\prime
},b^{\prime}\in\mathbb{R}$ such that
\[
a^{\prime}<b^{\prime}<a<b\text{ \quad and\quad\ }b^{\prime}-a^{\prime}\geq MR
\]
the curve $\gamma:[a^{\prime},b^{\prime}]\rightarrow{\mathbb{R}}^{3}$
with equation $\gamma(t)=\left(  r\,\cos t,\,r\,\sin t,\,\mu rt\right)  $,
satisfies
\[
C(\tau,\alpha)\cap\mathrm{Cyl\,}(R,[a,b])=\emptyset.
\]

\end{lemma}

\noindent\textit{Proof.} Without loss of generality, we can assume $b^{\prime
}=0$. The equation of the spiral $\gamma:[0,\infty)\rightarrow{\mathbb{R}}
^{3}$ is of the form
\[
\gamma(t)=\left(  r\,\cos t,\,r\,\sin t,\,\mu rt\right)  ,\qquad a^{\prime
}\leq t\leq0,
\]
where $\mu<1/2$ is given by Lemma~\ref{selfexpanded}. 
Set
\[
\sigma(\theta,z,u)=\left(  u\cos\theta,u\sin\theta,z\right)  ,\quad
\gamma(0)=\left(  r,0,0\right)  \quad\mbox{and}\quad\gamma'(0)=\left(
0,r,\mu r\right)  .
\]
It is enough to check that
\[
\langle\gamma'(0),\sigma(\theta,z,u)-\gamma(0)\rangle\leq\cos
\alpha\,\left\Vert\gamma'(0)\right\Vert \,\left\Vert \sigma
(\theta,z,u)-\gamma(0)\right\Vert \qquad\forall\theta\in\lbrack0,2\pi
],\,\forall z\in\lbrack a,b],\,\forall u\in\lbrack0,R].
\]
The above condition reads, for all $\theta
\in\lbrack0,2\pi]$, for all $z\in\lbrack a,b]$ and for all $u\in\lbrack0,R]$,
\[
ru\sin\theta+\mu rz\leq\frac{1}{\sqrt{5}}\,\sqrt{r^{2}(1+\mu^{2})}\,\sqrt{\left(
u\cos\theta-r\right)  ^{2}+u^{2}\sin^{2}\theta+z^{2}}.
\]
So it is enough to check that for all $z\in\lbrack
a,b]$ and $u\in\lbrack0,R]$ it holds
\[
u+\mu z\leq\frac{1}{\sqrt{5}}\,z\sqrt{1+\mu^{2}}.
\]
In order to do that, let us fix the value of $M$. 
For $\mu<1/2$, we have $\sqrt{1+\mu^{2}}>\sqrt{5}\mu$.
Therefore, we can choose $M>0$ such that $\sqrt{1+\mu^{2}}>\left(  \mu+\frac
{1}{M}\right)  \sqrt{5}$.
Now for all $u\leq R$ and $z\geq a$ we
have
\[
\frac{u+\mu z}{z}\leq\frac{R}{a}+\mu\leq\frac{1}{M}+\mu
<\frac{\sqrt{1+\mu^{2}}}{\sqrt{5}},
\]
This completes the proof of the lemma. \hfill$\square$

$\bigskip$

We are now ready to prove Theorem~\ref{main}, that is, given $\lambda
={\frac{1}{\sqrt{5}}}$, we construct a continuous curve $\gamma
:[0,+\infty]\rightarrow{\mathbb{R}}^{3}$ of infinite length, lying in the unit
ball, with nonzero right derivative at each point and satisfying the $\lambda
$-cone property ($\lambda$-eel).

\bigskip

\noindent\textit{Proof of Theorem~\ref{main}.}
We claim that we can construct a sequence of disjoint bounded cylinders
$$
C_{n}=\mathrm{Cyl\,}(r_{n},[a_{n},a_{n}+2\pi\mu r_n]),\quad n\ge 1
$$ 
along the $z$-axis, such
that $a_{n}\in\lbrack0,1)$, $r_{n+1}\leq r_{n}/2$,
$\ell_{n}:=a_{n}-(a_{n+1}+2\pi\mu r_{n+1})>0$ and $\ell_{n}/r_{n}$ is sufficiently
big to ensure that the cylinder $C_{n}$ is not seen by any $\lambda$-eel lying
in a (smaller) cylinder $C_{m}$ for $m>n$ (\textit{c.f.} Lemma
\ref{bigcylinder}). More precisely, we define $a_{0}=0$, and, for $n\geq1$,
\[
a_{n}=2^{-n}\text{\quad and\quad}r_{n}=\frac{1}{2^{n+1}(\pi\mu+M)},
\]
where $M>0$ is given by Lemma~\ref{bigcylinder}. Let us check that the
conditions of Lemma~\ref{bigcylinder} are fulfilled for the cylinders $C_{n}$
(big remote cylinder) and $C_{n+1}$ (small cylinder):
\[
\ell_{n}=a_{n}-(a_{n+1}+2\pi\mu r_{n+1})=\frac{1}{2^{n+1}}-\frac{2\pi\mu
}{2^{n+2}(\pi\mu+M)}=\frac{M}{2^{n+1}(\pi\mu+M)}\geq Mr_{n}.
\]
Now the construction is as follows. For each $n$, let $\gamma_{n}$ be the
$\lambda$-eel given by Proposition~\ref{ge1}, of length greater than $1$ lying
inside the cylinder $C_{n}$, entering this cylinder from the upper part
($z=a_{n}+2\pi\mu r_{n}$) and having its endpoint at the bottom ($z=a_{n}$).
Let $\tilde{e}_{n}$ be the oriented segment going from the endpoint of the
curve $\gamma_{n}$ (bottom of the cylinder $C_{n}$) to the starting point of
$\gamma_{n+1}$ (upper part of the cylinder $C_{n+1}$). We now define
$\gamma:[0,+\infty)\rightarrow{\mathbb{R}}^{3}$ by concatenation of the
following curves : $\gamma_{1}$, $\tilde{e}_{1}$, $\gamma_{2}$, $\tilde{e}
_{2}$, and so on. It is clear that $\gamma$ is continuous and has right
derivative at each point. Morever, $\gamma$ is contained in the unit ball of
${\mathbb{R}}^{3}$ and its length $\ell(\gamma)$ is greater than $\ell
(\gamma_{1})+\ell(\gamma_{2})+\cdots+\ell(\gamma_{n})\geq n$ for every $n$,
therefore it is infinite. Observe that $\gamma(t)$ has limit $0$ as
$t\rightarrow+\infty$. It remains to prove that $\gamma$ is a $\lambda$-eel,
that is, it satisfies the $\lambda$-cone condition. Notice that each curve
$\gamma_{n}$, $\tilde{e}_{n}$ is individually a $\lambda$-eel (that is, it satisfies the $\lambda$-cone
property with respect to itself). Provided $M$ is sufficiently big, the segment $\tilde{e}_{n}$ 
is almost parallel to the $z$-axis and it is oriented to the opposite direction of the previous curves $\gamma_{1}$,
$\tilde{e}_{1}$,$\cdots$, $\gamma_{n}$. Therefore, if the $\lambda$-cone $C(t,\alpha)$, given in \eqref{simple}, has its origin onto a segment $\tilde{e}_{n}$, then it does not meet the union of the ranges of $\gamma_{1}$, $\tilde{e}_{1}$,$\cdots$, $\gamma_{n}$. It remains to treat the case where $C(t,\alpha)$ has its origin to a curve of the form $\gamma_{n}$. These curves are constructed (for each $n$) by concatenating pieces of the form $\gamma_{i}^{\downarrow}$ (of type~1), $\gamma
_{i}^{\uparrow}$ (of type~2) and $e_{i}^{+}$ or $e_{i}^{-}$ (of type~3) (\textit{c.f.} Remark~\ref{type}).
If the $\lambda$-cone lies on a piece of type~1 or of type~3 of $\gamma_n$, then it is oriented to the opposite directions of all of the previous pieces $\gamma_{1}$, $\tilde{e}_{1}$,$\cdots$, $\gamma_{n-1}$, $\tilde{e}_{n-1}$ of $\gamma$, therefore it does not meet the union of their ranges. If now the cone $C(t,\alpha)$ has its origin on an upward piece $\gamma_{i}^{\uparrow}$ (type~2) of the curve $\gamma_{n}$, then the result follows
from Lemma~\ref{bigcylinder}. The proof is complete. \hfill$\square$

%%%%%%%%
% Section 4
\section{Curves with the $\lambda$-cone property in 2 dimensions}

It is remarkable that there is no analogue of the construction in Theorem~\ref{main} in dimension~2. Indeed, we shall show that for any value of the parameter $\lambda \in \lbrack -1,1)$, any bounded planar $\lambda $-eel
(that is, continuous curve with right derivative at each point that satisfies the $\lambda $-cone property) is rectifiable and has finite
length. We shall need the following lemmas. (Recall $\alpha =\arccos(\lambda )$.)

\begin{lemma}
\label{lemma-snake-1}Let $\gamma :I\rightarrow {\mathbb{R}}^{2}$ be a planar 
$\lambda $-eel and $t_{1}<t_{2}<t_{3}$ in $I.$ Then 
\begin{equation*}
\gamma (t_{3})\notin \lbrack \gamma (t_{1}),\gamma (t_{2})].
\end{equation*}
\end{lemma}

\noindent \textit{Proof.} Set $A=\gamma (t_{1}),$ $B=\gamma (t_{2})$, $C=\gamma (t_{3})$ and assume towards a contradiction that $C\in \lbrack A,B]$. Choosing adequate coordinates in $\mathbb{R}^{2}$ we may assume that 
$A=(0,0)$, $B=(1,0)$ and $C=(c,0)$ with $c\in (0,1).$ In the sequel, we shall
write $\gamma =(\gamma _{1},\gamma _{2})$ in these coordinates.\smallskip 

Before we proceed, notice that we may assume
\begin{equation}
\gamma (t)\notin (A,C)\text{\quad for all }t\in (t_{1},t_{2}]\text{.}
\label{ad}
\end{equation}
Indeed, set $N_{1}=\{t\in \lbrack t_{1},t_{2}):\,\gamma (t)\in \lbrack
A,C]\}=\{t\in \lbrack t_{1},t_{2}):\,\gamma _{1}(t)\in \lbrack 0,c],\,\gamma
_{2}(t)=0\}$ and $\alpha _{1}:=\sup \{\gamma _{1}(t):t\in N_{1}\}$. Then 
$\alpha _{1}<c$ (since $\gamma $ is continuous and injective) and
consequently, there exists $t_{1}\leq \tilde{t}_{1}<t_{2}$ with $\gamma (\tilde{t}_{1})=(\alpha _{1},0)=\tilde{A}.$ In this case we can replace $A$
by $\tilde{A}$ and $t_{1}$ by $\tilde{t}_{1}$ and get \eqref{ad}.\smallskip 

We set $\tilde{t}_{2}=\inf \{t\in \lbrack t_{1},t_{2}]:\,\gamma _{1}(t)\geq
c,\,\gamma _{2}(t)=0\}.$ There is no loss of generality to assume $t_{2}=
\tilde{t}_{2},$ since we can always replace $B$ by $\tilde{B}=(\gamma _{1}(\tilde{t}_{2}),0)$ (notice that $\gamma _{1}(\tilde{t}_{2})>c$ by
injectivity).\smallskip 

Therefore for all $t\in (t_{1},t_{2})$ we have $\gamma (t)\notin (A,B).$
Setting $\Gamma _{AB}=\{\gamma (t):t\in \lbrack t_{1},t_{2}]\}$ we deduce
that $\Gamma _{AB}\cup (B,A]$ is a Jordan curve which separates $\mathbb{R}
^{2}$ in two regions, exactly one of them being bounded. Call $\mathcal{R}$
this bounded region, set $H^{+}=\{x=(x_{1},x_{2}):x_{2}>0\}$, $H^{-}=\{x=(x_{1},x_{2}):x_{2}<0\}$ and let $\varepsilon >0$ be such that $B(C,\varepsilon )\cap \Gamma _{AB}=\emptyset .$ Then at least one of the
sets $B(C,\varepsilon )\cap H^{+}$ and $B(C,\varepsilon )\cap H^{-}$ has
nonempty intersection with $\mathcal{R}.$ Assume, with no loss of
generality, that
\begin{equation*}
B(C,\varepsilon )\cap H^{-}\cap \mathcal{R}\neq \emptyset .
\end{equation*}
Then for every $x\in H^{-}\cap \mathrm{int\,}\mathcal{R}$ and every
direction $d=(d_{1},d_{2})\in \mathbb{S}^{1}$ (the unit sphere of $\mathbb{R
}^{2}$) with $d_{2}\leq 0$ it holds $\ell _{x,d}\cap \Gamma _{AB}\neq
\emptyset ,$ where $\ell _{x,d}:=\{x+\mu d:\mu \geq 0\}$ is the half-line
emanating from $x$ with direction $d.$ In particular, shrinking $\varepsilon >0$
if necessary, and recalling notation (\ref{cone-notation}) we deduce that 
\begin{equation}
C_{x}(d,\alpha )\cap \Gamma _{AB}\neq \emptyset ,\quad \text{for all }x\in
B(C,\varepsilon )\cap H^{-}\cap \mathcal{R\;}\text{and all\ }
d=(d_{1},d_{2})\;\text{with }d_{2}\leq 0.  \label{ad1}
\end{equation}
Let $\tau _{3}\in (t_{2},t_{3})$ be such that for all $t\in (\tau _{3},t_{3}]
$ we have $\gamma (t)\in B(C,\varepsilon )$ (such $\tau _{3}$ exists by
continuity). Then it follows by (\ref{ad1}) and the $\lambda $-eel property
that $\gamma'_{2}(t)>0,$ and consequently, $\gamma _{2}(t)<0$ (since $
\gamma _{2}(t_{3})=0$). Let further $\tau \in \lbrack t_{2},t_{3}]$ be such
that
\begin{equation*}
\gamma _{2}(\tau )=\min_{t\in \lbrack t_{2},t_{3}]}\gamma _{2}(t)\;(\,<0\,).
\end{equation*}
Then since $\gamma _{2}(t_{2})=0,$ there exists $\tilde{t}\in \lbrack
t_{2},\tau ]$ with ($\gamma (t)\in \mathcal{R}$ and) $\gamma'_{2}(t)<0$
which together with (\ref{ad1}) contradicts the $\lambda $-eel
property.\hfill $\square $

\bigskip 

For the next statement, recall notation \eqref{simple} and \eqref{cone-notation}.

\begin{lemma}
\label{lemma-snake-2}Under the assumptions of the previous lemma we have: 
\begin{equation*}
C(\gamma'(t),\alpha)\cap K(t)=\{0\},\quad \text{for all }t\in I.
\end{equation*}
\end{lemma}

\noindent \textit{Proof.} Fix $t\in I$ and assume with no loss of generality
(by translation) that $\gamma (t)=0.$ Then $K(t)=\overline{\mathrm{cone}}
\mathrm{\,}(\Gamma (t)-\gamma (t))=\overline{\mathrm{cone}}\mathrm{\,}\Gamma
(t),$ where $\Gamma (t)=\{\gamma (\tau ):t\in \lbrack 0,t]\}.$ Let assume
that there exists $x\in C(\gamma'(t),\alpha)\cap K(t),$ $x\neq 0.$ Then by
Caratheodory theorem, there exist $x_{i}=\gamma (\tau _{i}),$ $i\in \{1,2,3\}
$ with $\tau _{1}\leq \tau _{2}\leq \tau _{3}<t$ and $x\in \mathrm{conv}
\{x_{1},x_{2},x_{3}\}$ (convex envelope). Set $\ell _{1}:=\{x_{1}+\mu
(x-x_{1}):\,\mu \geq 0\}$  and $\ell _{2}=\{x_{2}+\mu (x-x_{2}):\,\mu \geq
0\}.$ If $\ell _{1}\cap \Gamma (t)=\ell _{2}\cap \Gamma (t)=\emptyset ,$
then for $\mu _{1},$ $\mu _{2}$ sufficiently big, the point $x_{3}$ should
belong to the triangle defined by the points $\ell _{1}(\mu _{1}):=x_{1}+\mu
_{1}(x-x_{1}),$  $x$ and $\ell _{2}(\mu _{2})=x_{2}+\mu _{2}(x-x_{2}).$ Then
by connectedness of $\gamma ([\tau _{1},\tau _{3}])$, we deduce that for
some $s<\tau _{3}<t$ it holds $\gamma (s)\in \ell _{1}\cup \ell _{2}$. We
deduce that $x$ is a convex combination of two points of $\Gamma (t),$ that
is, $x\in \lbrack \gamma (s_{1}),\gamma (s_{2})]$ for some $s_{1}<s_{2}<t.$
Set $\Gamma _{12}:=\{\gamma (\tau \}:\tau \in \lbrack s_{1},s_{2}]\}.$ Since 
$\gamma $ is a $\lambda $-eel, we have $C(t,\alpha)\cap \Gamma _{12}=\emptyset .$
Then $[\gamma (s_{2}),\gamma (s_{1})]\cup \Gamma _{12}$ is a Jordan curve
and $\gamma (t)=0\in \mathcal{R}$ where $\mathcal{R}$ is the bounded region
delimited by the Jordan curve. This yields that for some $t_{1}<t_{2}<t$, $\gamma (t)=0$ is a convex combination of $\gamma (t_{1})$ and $\gamma(t_{2}),$ which contradicts Lemma~\ref{lemma-snake-1}.\hfill $\square $

\bigskip 

In view of Lemma~\ref{lemma-snake-2} and Remark~\ref{essence} we obtain
our main result.

\begin{theorem}[bounded planer eels have finite length]
\label{prop-2-dim}Let $\gamma :I\rightarrow {\mathbb{R}}^{2}$ be a bounded 
$\lambda $-eel. Then $\gamma$ is rectifiable and has finite length.\hfill$\square$
\end{theorem}

\bigskip

\textbf{Acknowledgment.} The authors wish to thank Ludovic Rifford for several useful
discussions.

\vspace{0.8cm}

\noindent Aris Daniilidis

\smallskip

\noindent DIM--CMM, UMI CNRS 2807\newline Beauchef 581, Torre Norte, piso 5,
Universidad de Chile\newline Santiago CP8370456, Chile \medskip

\noindent E-mail: \texttt{arisd@dim.uchile.cl} \newline\noindent
\texttt{http://www.dim.uchile.cl/\symbol{126}arisd}

\medskip

\noindent Research supported by the grants: \newline BASAL PFB-03, FONDECYT
1171854, ECOS/CONICYT C14E06, REDES/CONICYT 15040 (Chile) and
MTM2014-59179-C2-1-P (MINECO of Spain and ERDF of EU).\newline

\vspace{0.45cm}

\noindent Robert Deville

\smallskip

\noindent Laboratoire Bordelais d'Analyse et Geom\'{e}trie\newline Institut de
Math\'{e}matiques de Bordeaux, Universit\'{e} de Bordeaux 1\newline351 cours
de la Lib\'{e}ration, Talence Cedex 33405, France \medskip

\noindent E-mail: \texttt{Robert.Deville@math.u-bordeaux1.fr}

\medskip

\noindent Research supported by the grants: \newline ECOS/CONICYT C14E06
(France) and REDES/CONICYT-15040 (Chile).

\vspace{0.45cm}

\noindent Estibalitz Durand-Cartagena

\smallskip

\noindent Departamento de Matem\'{a}tica Aplicada\newline ETSI Industriales,
UNED\newline Juan del Rosal 12, Ciudad Universitaria, E-28040 Madrid, Spain
\medskip

\noindent E-mail: \texttt{edurand@ind.uned.es}\newline\noindent
\texttt{http://www.uned.es/personal/edurand} \newline

\medskip

\noindent Research supported by the grant MTM2015-65825-P (MINECO of Spain) and 2018-MAT14 (ETSI Industriales, UNED).

\end{document}